\documentclass[11pt]{article}

\usepackage{amssymb, amsmath, amsthm, graphicx}
\usepackage[left=1in,top=1in,right=1in]{geometry}

\usepackage{subfigure}

\theoremstyle{plain}
      \newtheorem{theorem}{Theorem}[section]
      \newtheorem{lemma}[theorem]{Lemma}
      
      \newtheorem{observation}[theorem]{Observation}
      \newtheorem{corollary}[theorem]{Corollary}

\theoremstyle{definition}

\theoremstyle{remark}

	\newcommand{\RR}{{\mathbb R}}

\def\sgn{\mbox{\rm sgn}}
\def\conv{\mbox{\rm conv}}

\title{A polynomial regularity lemma for semi-algebraic hypergraphs and its applications in geometry and property testing\footnote{A preliminary version of this paper appeared in SODA 2015 \cite{FoPaSu}.}}
\author{Jacob Fox \thanks{Stanford University, Stanford, CA. Supported by a Packard Fellowship, by NSF CAREER award DMS 1352121, and by an Alfred P. Sloan Fellowship. Email: {\tt jacobfox@stanford.edu.}} \and J\'anos Pach\thanks{EPFL, Lausanne and Courant Institute, New York, NY. Supported
by a Hungarian Science Foundation NKFI grant, by Swiss National Science Foundation Grants 200021-165977 and 200020-162884.  Email:
{\tt pach@cims.nyu.edu}.}\and Andrew Suk\thanks{University of Illinois at Chicago, Chicago, IL.  Supported by NSF grant DMS-1500153.   Email: {\tt suk@uic.edu}.} }

\begin{document}

\maketitle

\begin{abstract}

In this paper, we prove several extremal results for geometrically defined hypergraphs.  In particular, we establish an improved lower bound, single exponentially decreasing in $k$, on the best constant $\delta>0$ such that the vertex classes $P_1,\ldots,P_k$ of every {\em $k$-partite} $k$-uniform semi-algebraic hypergraph $H=(P_1\cup\ldots\cup P_k, E)$ with $|E|\ge\varepsilon\Pi_{j=1}^k|P_i|$ have, for $1 \leq i \leq k$, $\delta|P_i|$-element subsets $P'_i\subseteq P_i$ satisfying  $P_{1}'\times\ldots\times P_{k}'\subseteq E$. The best previously known lower bound on $\delta$ due to Bukh and Hubard decreased triple exponentially fast in $k$. We give three geometric applications of our results. In particular, we establish the following strengthening of the so-called same-type lemma of B\'ar\'any and Valtr: Any disjoint finite sets $P_1,\ldots,P_k\subset \mathbb{R}^d\; (k>d)$ have, for $1 \leq i \leq k$, subsets $P'_i$ of size at least $2^{-O(d^3k\log k)}|P_i|$ with the property that every $k$-tuple formed by taking one point from each $P'_i$ has the same order type.

We also improve a result of Fox, Gromov, Lafforgue, Naor, and Pach, who established a regularity lemma for semi-algebraic $k$-uniform hypergraphs of bounded complexity, showing that for each $\varepsilon>0$ the vertex set can be equitably partitioned into a bounded number of parts (in terms of $\varepsilon$ and the complexity) so that all but an $\varepsilon$-fraction of the $k$-tuples of parts are homogeneous. Here, we prove that the number of parts can be taken to be polynomial in $1/\varepsilon$. Our improved regularity lemma can be applied to geometric problems and to the following general question on property testing: is it possible to decide, with query complexity polynomial in the reciprocal of the approximation parameter, whether a hypergraph has a given hereditary property? We give an affirmative answer for testing typical hereditary properties for semi-algebraic hypergraphs of bounded complexity.

\end{abstract}

\section{Introduction}

It is known in Ramsey theory that in any bipartite graph $G = (U\cup V,E)$ with $|U| = |V| = n$, there are two large subsets $U'\subset U$, $V'\subset V$ with $|U'|,|V'| \geq \Omega(\log n)$ such that either every vertex in $U'$ is connected to every vertex in $V'$, or no vertex in $U'$ is connected to any vertex in $V'$.  In other words, we have either $U'\times V' \subset E$ or $U'\times V' \cap E = \emptyset$.  However, much stronger theorems are known for geometrically defined bipartite graphs.  For instance, it was shown in \cite{PS5} that if $G$ is an intersection graph of two $n$-element sets of segments, $U$ and $V$, in the plane, then one can find subsets $U'\subset U$ and $V'\subset V$ with the above properties such that $|U'|, |V'| \geq \Omega(n)$. Moreover, in this latter example, it is also true that for a fixed $\varepsilon > 0$ if $|E(G)|\geq \varepsilon n^2$, then we can find $U'\subset U$ and $V'\subset V$ with $|U'| = |V'|= \Omega(n)$, such that $U'\times V' \subset E$ (here the hidden constant depends on $\varepsilon$).

The above results have been generalized to hypergraphs.  A $k$-uniform hypergraph $H = (P,E)$ consists of a vertex set $P$ and an edge set (or hyperedge set) $E$, which is a collection of $k$-element subsets of $P$.  A hypergraph is {\it $k$-partite} if it is $k$-uniform and its vertex set $P$ is partitioned into $k$ parts, $P=P_1\cup\ldots\cup P_k$, such that every edge has precisely one vertex in each part. It follows from a classical theorem of Erd\H os~\cite{erdos1}, which was one of the starting points in extremal hypergraph theory, that if $|P_1|=\ldots=|P_k|$ and $|E|\ge \varepsilon\Pi_{i=1}^k|P_i|$ for some $\varepsilon>0$, then one can find subsets $P'_i\subset P_i$ such that
$$|P'_i| = \Omega\left(\frac{\log |P_i|}{\log (1/\varepsilon)}\right)^{1/(k-1)},$$
and $P'_1\times \cdots \times P'_k\subset E$. In other words, $H$ contains a large {\em complete} $k$-partite subhypergraph.

Just like for bipartite graphs, larger complete $k$-partite subhypergraphs can be found in geometrically defined hypergraphs.  In particular, this is the case for hypergraphs that admit a simple algebraic description.  To make this statement precise, we need some terminology.

\medskip

\noindent{\bf Semi-algebraic setting.} A $k$-partite hypergraph $H = (P_1\cup\ldots\cup P_k,E)$ is called \emph{semi-algebraic} in $\RR^d$, if its vertices are points in $\RR^d$, and there are polynomials $f_1,\ldots,f_t \in \mathbb{R}[x_1,\ldots,x_{kd}]$ and a Boolean function $\Phi$ such that for every $(p_1,\ldots, p_k) \in P_1\times \cdots \times P_k$, we have
$$(p_1,\ldots,p_k) \in E \hspace{.5cm}\Leftrightarrow\hspace{.5cm} \Phi(f_1(p_1,\ldots,p_k) \geq 0;\ldots;f_t(p_1,\ldots,p_k) \geq 0) = 1.$$
At the evaluation of $f_j(p_1,\ldots,p_k)$, we substitute the variables $x_1,\ldots,x_k$ with the coordinates of $p_1$, the variables $x_{k+1},\ldots,x_{2k}$ with the coordinates of $p_2$, etc.

We say that $H$ has \emph{complexity} $(t,D)$ if each polynomial $f_j$ with $1 \leq j \leq t$ has the property that for any fixed $k-1$ points $q_1,\ldots,q_{k-1} \in \mathbb{R}^d$, the $d$-variate polynomials

$$\begin{array}{ccl}
  h_{j,1}(\textbf{x}_1) & = & f_j(\textbf{x}_1,q_1,\ldots,q_{k-1}), \\
  h_{j,2}(\textbf{x}_2) & = & f_j(q_1,\textbf{x}_2,q_2,\ldots,q_{k-1}), \\
   & \vdots &  \\
  h_{j,k}(\textbf{x}_k) & = & f_j(q_1,\ldots,q_{k-1},\textbf{x}_k)
\end{array}$$

\noindent are of degree at most $D$ (in notation, $deg(h_{j,i}) \leq D$ for $1\le j\le t$ and $1\le i\le k$). It follows that $deg(f_j) \leq kD$ for every $j$.

If our $k$-uniform hypergraph $H = (P,E)$ is {\em a priori} not $k$-partite, we will assume that its relation $E$ is \emph{symmetric}.  More precisely, for any fixed enumeration $p_1, p_2,\ldots$ of the elements of $P\subset\mathbb{R}^d$, we say that $H$ is {\em semi-algebraic} with complexity $(t,D)$ if for every $1\leq i_1 < \cdots < i_k \leq n$, $(p_{i_1},\ldots,p_{i_k}) \in E$ iff for all permutation $\pi$,

$$ \Phi(f_1(p_{\pi(i_1)},\ldots,p_{\pi(i_k)}) \geq 0,\ldots, f_t(p_{\pi(i_1)},\ldots,p_{\pi(i_k)}) \geq 0) = 1,$$

\noindent where $\Phi$ is a Boolean function and $f_1,\ldots,f_t$ are polynomials satisfying the same properties as above.

\medskip
\noindent{\bf Density theorem for semi-algebraic hypergraphs.}
Fox {\em et al.}~\cite{gromov} showed that there exists a constant $c = c(k,d,t,D)>0$ with the following property. Let $(P_1\cup\ldots\cup P_k,E)$ be any $k$-partite semi-algebraic hypergraph in $\RR^d$ with complexity $(t,D)$, and suppose that $|E|\ge \varepsilon\Pi_{i=1}^k|P_i|$. Then one can find subsets $P'_i \subset P_i$, $1\leq i \leq k$, with $|P'_i| \geq \varepsilon^{c}|P_i|$, which induce a complete $k$-partite subhypergraph, that is, $P'_1\times \ldots \times P'_k \subset E$. The original proof gives a poor upper bound on $c(k,d,t,D)$, which is of tower-type in $k$.  Combining a result of Bukh and Hubard \cite{hubard} with a variational argument of Koml\'os~\cite{sz}, see also Section 9.4 in \cite{matousek}, the dependence on $k$ can be improved to double exponential. Our following result, which will be proved in Section 2, removes the dependency on $k$ in the exponent of $\varepsilon$.

\begin{theorem}\label{main0}

For any positive integers $d,t,D$, there exists a constant $C=C(d,t,D) > 0$ with the following property. Let $\varepsilon>0$ and let $H = (P_1\cup\ldots\cup P_k,E)$ be any $k$-partite semi-algebraic hypergraph in $\RR^d$ with complexity $(t,D)$ and $|E| \geq \varepsilon\Pi_{i=1}^k|P_i|$. Then one can choose subsets $P'_i \subset P_i$, $1\leq i \leq k$, such that

$$|P'_i| \geq \frac{\varepsilon^{{d + D \choose d}}}{C^k}|P_i|,$$

\noindent and $P'_1\times \ldots \times P'_k \subseteq E$.  Moreover, we can take $C = 2^{20m\log (m+1)}t^{m/k}$, where $m =  {d + D\choose d} - 1$.

\end{theorem}

The main novelty in the proof of Theorem \ref{main0} is that we completely avoid using the arguments of Koml\'os \cite{sz}, which were used in \cite{noga} and \cite{gromov}, in his proof of a variant of Szemer\'edi's regularity lemma.  Instead, we combine an inductive argument on $k$ with an old result on cell decomposition in order to remove the dependency on $k$ in the exponent of $\varepsilon$.

For the applications given in Section \ref{app}, the dependence on the dimension $d$ becomes crucial. This is typically the case for relations that have complexity $(t,1)$ (i.e., when $D = 1$).  Substituting $D=1$ in Theorem \ref{main0}, we obtain the following.

\begin{corollary}\label{main2}
Let $\varepsilon>0$ and let $H = (P_1\cup\ldots\cup P_k,E)$ be a $k$-partite semi-algebraic hypergraph in $\mathbb{R}^d$ with complexity $(t,1)$ and $|E| \geq \varepsilon\Pi_{i=1}^k|P_i|$. Then one can choose subsets $P'_i \subset P_i$, $1\leq i \leq k$, such that

$$|P'_i| \geq \frac{\varepsilon^{d + 1}}{2^{20kd\log (d + 1)}t^{d }}|P_i|,$$

\noindent and $P'_1\times \cdots \times P'_k \subseteq E$.
\end{corollary}

In the other direction, Karasev, Kyn\v cl, Pat\'ak, Pat\'akov\'a, and Tancer \cite{kyncl} constructed a $(d+1)$-partite $(d+1)$-uniform semi-algebraic hypergraph $H = (P_1 \cup \cdots \cup P_{d+1},E)$ in $\RR^d$ with complexity $(d + 1,1)$ and $|E| \geq \varepsilon\Pi_{i=1}^{d+1}|P_i|$, where $\varepsilon = 2^{-cd}$ for a fixed constant $c > 0$, and $H$ has the following property.  For any subsets $P_1' \subset P_1,\ldots, P'_{d + 1}\subset P_{d + 1}$ such that $P_{1}'\times\ldots\times P_{d+1}'\subseteq E$, we have $|P'_i| < 2^{-c'd\log d}|P_i|$, with a suitable constant $c' > 0$.  Note that in this case, Corollary \ref{main2} implies the existence of such subsets $P'_1,\ldots, P'_{d +1}$ with $|P'_i| \geq 2^{-c''d^2\log d}|P_i|$ for some $c'' > 0$, so there is still a gap between the lower and upper bounds.

\medskip

\noindent{\bf Polynomial regularity lemma for semi-algebraic hypergraphs.} Szemer\'edi's regularity lemma is one of the most powerful tools in modern combinatorics. In its simplest version~\cite{szemeredi} it gives a rough structural characterization of all graphs. A partition is called {\em equitable} if any two parts differ in size by at most one. According to the lemma, for every $\varepsilon>0$ there is $K=K(\varepsilon)$ such that every graph has an equitable vertex partition into at most $K$ parts such that all but at most an $\varepsilon$ fraction of the pairs of parts behave ``regularly".\footnote{For a pair $(P_i,P_j)$ of vertex subsets, $e(P_i,P_j)$ denotes the number of edges in the graph running between $P_i$ and $P_j$. The density $d(P_i,P_j)$ is defined as $\frac{e(P_i,P_j)}{|P_i||P_j|}$. The pair $(P_i,P_j)$ is called $\varepsilon$-regular if for all $P_i' \subset P_i$ and $P_j' \subset P_j$ with $|P_i'| \geq \varepsilon |P_i|$ and $|P_j'| \geq \varepsilon |P_j|$, we have $|d(P'_i,P'_j)-d(P_i,P_j)| \leq \varepsilon$.} The dependence of $K$ on $1/\varepsilon$ is notoriously bad. It follows from the proof that $K(\varepsilon)$ may be taken to be of an exponential tower of $2$-s of height $\varepsilon^{-O(1)}$. Gowers \cite{Go97} used a probabilistic construction to show that such an enormous bound is indeed necessary. Consult \cite{CoFo}, \cite{MoSh}, \cite{FoLo} for other proofs that improve on various aspects of the result. Szemer\'edi's regularity lemma was extended to $k$-uniform hypergraphs by Gowers \cite{Go1,Go2} and by Nagle \emph{et al}. \cite{NRSS}. The bounds on the number of parts go up in the Ackermann hierarchy, as $k$ increases. This is quite unfortunate, because in property testing and in other algorithmic applications of the regularity lemma this parameter has a negative impact on the efficiency.

Alon~{\em et al.}~\cite{noga} (for $k=2$) and Fox~{\em et al.}~\cite{gromov} (for $k>2$) established an ``almost perfect" regularity lemma for $k$-uniform semi-algebraic hypergraphs $H=(P,E)$. According to this, $P$ has an equitable partition such that all but at most an $\varepsilon$-fraction of the $k$-tuples of parts $(P_{i_1},\ldots,P_{i_k})$ behave not only regularly, but {\em homogeneously} in the sense that either $P_{i_1}\times\ldots\times P_{i_k}\subseteq E$ or $P_{i_1}\times\ldots\times P_{i_k}\cap E=\emptyset$. The proof is essentially qualitative: it gives a very poor estimate for the number of parts in such a partition.

In Section \ref{structure}, we deduce a much better quantitative form of this result, showing that the number of parts can be taken to be polynomial in $1/\varepsilon$.

\begin{theorem}
\label{reg1}
For any positive integers $k, d, t, D$ there exists a constant $c=c(k,d,t,D)>0$ with the following property. Let $0<\varepsilon < 1/2$ and let $H=(P,E)$ be a $k$-uniform semi-algebraic hypergraph in $\mathbb{R}^d$ with complexity $(t,D)$. Then $P$ has an equitable partition $P = P_1\cup \cdots \cup P_K$  into $K\leq (1/\varepsilon)^c$ parts such that all but an $\varepsilon$-fraction of the $k$-tuples of parts are homogeneous.
\end{theorem}

The aim of this paper is to prove Theorems \ref{main0} and \ref{reg1}, and to apply them to extremal and algorithmic questions in combinatorial geometry.  In the following two subsections, we outline these applications. See \cite{conlon}, for other favorable Ramsey-type properties of semi-algebraic sets.

\medskip

\noindent \textbf{Geometric applications.} In Section~\ref{app}, we prove three geometric applications of Corollary \ref{main2}.

\medskip

\noindent \emph{1. Same-type lemma.}  Let $P = (p_1,\ldots,p_n)$ be an $n$-element point sequence in $\mathbb{R}^d$ in {\em general position}, i.e., assume that no $d+1$ points lie in a common hyperplane. For $i_1 < i_2 < \cdots < i_{d+1}$, the \emph{orientation} of the $(d+1)$-tuple $(p_{i_1},p_{i_2},\ldots,p_{i_{d+1}}) \subset P$ is defined as the sign of the determinant of the unique linear mapping $M$ that sends the $d$ vectors $p_{i_2}-p_{i_1},p_{i_3}-p_{i_1},\ldots,p_{i_{d+1}} - p_{i_1}$, to the standard basis $e_1,e_2,\ldots,e_{d}$. Denoting the coordinates of $p_i$ by $x_{i,1},\ldots,x_{i,d}$, the orientation of $(p_{i_1},p_{i_2},\ldots,p_{i_{d+1}})$ is
 $$\chi = \sgn\det\left(\begin{array}{cccc}
                                                           1 & 1 & \cdots & 1 \\
                                                           x_{1,1} & x_{2,1} & \cdots & x_{d+1,1} \\
                                                           \vdots & \vdots & \vdots & \vdots \\
                                                           x_{1,d} & x_{2,d} & \cdots & x_{d+1,d}
                                                         \end{array}\right),$$

 \noindent where $\sgn(r)$ denotes the sign of $r$ in $\RR$.  The \emph{order-type} of $P= (p_1,p_2,\ldots,p_n)$ is the mapping $\chi:{P \choose d + 1}\rightarrow \{+1,-1\}$ (positive orientation, negative orientation), assigning each $(d+1)$-tuple of $P$ its orientation. Therefore, two $n$-element point sequences $P$ and $Q$ have the same order-type if they are ``combinatorially equivalent."  See \cite{matousek} and \cite{gp} for more background on order-types.

Let $(P_1,\ldots,P_k)$ be a $k$-tuple of finite sets in $\mathbb{R}^d$.  A \emph{transversal} of $(P_1,\ldots,P_k)$ is a $k$-tuple $(p_1,\ldots,p_k)$ such that $p_i \in P_i$ for all $i$.  We say that $(P_1,\ldots,P_k)$ has \emph{same-type transversals} if all of its transversals have the same order-type.  B\'ar\'any and Valtr \cite{barany} showed that for $d,k > 1$, there exists a $c = c(d,k)$ such that the following holds.  Let $P_1,\ldots,P_k$ be finite sets in $\mathbb{R}^d$ such that $P_1\cup\cdots \cup P_k$ is in general position.  Then there are subsets $P'_1\subset P_1,\ldots,P'_k\subset P_k$ such that the $k$-tuple $(P'_1,\ldots,P'_k)$ has same-type transversals and $|P'_i| \geq c(d,k)|P_i|$.  Their proof shows that $c(d,k) = 2^{-k^{O(d)}}$.  We make the following improvement.

\begin{theorem}\label{same}
For $k > d$, let $P_1,\ldots,P_k$ be finite sets in $\mathbb{R}^d$ such that $P_1\cup\cdots \cup P_k$ is in general position.  Then there are subsets $P'_1\subset P_1,\ldots,P'_k\subset P_k$ such that the $k$-tuple $(P'_1,\ldots,P'_k)$ has same-type transversals and

$$|P_i'| \geq 2^{-O(d^3k \log k)}|P_i|,$$

\noindent for all $i$.

\end{theorem}

\medskip

\noindent \emph{2. Homogeneous selections from hyperplanes.}  B\'ar\'any and Pach \cite{barany2} proved that for every integer $d\geq 2$, there is a constant $c = c(d) > 0$ with the following property.  Given finite families $L_1,\ldots,L_{d+1}$ of hyperplanes in $\mathbb{R}^d$ in general position,\footnote{No element of $\bigcup_{i=1}^{d+1} L_i$ passes through the origin, any $d$ elements have precisely one point in common, and no $d+1$ of them have a nonempty intersection.} there are subfamilies $L'_i\subset L_i$ with $|L'_i| \geq c(d)|L_i|$, $1 \leq i \leq d+1$, and a point $q \in \mathbb{R}^d$ such that for every $(h_1,\ldots,h_{d+1}) \in L'_1\times \cdots \times L'_{d+1}$, point $q$ lies in the unique bounded simplex $\Delta(h_1,\ldots,h_{d+1})$ enclosed by $\bigcup_{i = 1}^{d+1} h_i$.  The proof gives that one can take $c(d) = 2^{-(d+1)2^d}$, because they showed that $c(d) = c(d,d+2)$ will meet the requirements, where $c(d,k)$ denotes the constant defined above for same-type transversals. Thus, Theorem \ref{same} immediately implies the following improvement.

\begin{theorem}
Given finite families $L_1,\ldots,L_{d+1}$ of hyperplanes in $\mathbb{R}^d$ in general position, there are subfamilies $L'_i\subset L_i$, $1 \leq i \leq d+1$, with

$$|L'_i| \geq 2^{-O(d^4\log (d+1))}|L_i|,$$

\noindent and a point $q \in \mathbb{R}^d$ such that for every $(h_1,\ldots,h_{d+1}) \in L'_1\times \cdots \times L'_{d+1}$, $q$ lies in the unique bounded simplex $\Delta(h_1,\ldots,h_{d+1})$ enclosed by $\bigcup_{i=1}^{d+1} h_i$.

\end{theorem}

\medskip

\noindent \emph{3. Tverberg-type result for simplices.}  In 1998, Pach \cite{pach2} showed that for all natural numbers $d$, there exists $c' = c'(d)$ with the following property.  Let $P_1,P_2,\ldots,P_{d+1}\subset \mathbb{R}^d$ be disjoint $n$-element point sets with $P_1\cup \cdots \cup P_{d+1}$ in general position.  Then there is a point $q\in \mathbb{R}^d$ and subsets $P'_1\subset P_1,\ldots,P'_{d+1} \subset P_{d+1}$, with $|P'_i| \geq c'(d)|P_i|$, such that all closed simplices with one vertex from each $P'_i$ contains $q$.  The proof shows that $c'(d) = 2^{-2^{2^{O(d)}}}$.  Recently, Karasev et al.~\cite{kyncl} improved this to $c'(d) > 2^{-2^{d^2 + O(d)}}$.  Here we make the following improvement.

\begin{theorem}\label{pach}
Let $P_1,P_2,\ldots,P_{d+1}\subset \mathbb{R}^d$ be disjoint $n$-element point sets with $P_1\cup \cdots \cup P_{d+1}$ in general position.  Then there is a point $q\in \mathbb{R}^d$ and subsets $P'_1\subset P_1,\ldots,P'_{d+1} \subset P_{d+1}$, with

$$|P'_i| \geq 2^{-O(d^2\log (d+1))}|P_i|,$$

\noindent such that all closed simplices with one vertex from each $P'_i$ contains $q$.
\end{theorem}

\medskip

\noindent{\bf Algorithmic applications: Property testing.}  The goal of property testing is to quickly distinguish between objects that satisfy a certain property and objects that are far from satisfying that property. This is an active area of computer science which was initiated by Rubinfeld and Sudan \cite{RuSu}. Subsequently, Goldreich, Goldwasser, and Ron \cite{GGR} started the investigation of property testers for combinatorial objects.

Graph property testing, in particular, has attracted a great deal of attention. A {\em property} $\mathcal{P}$ is a family of graphs closed under isomorphism. A graph $G$ with $n$ vertices is {\em $\epsilon$-far} from satisfying $\mathcal{P}$ if one must change the adjacency relation of at least an $\epsilon$ fraction of all pairs of vertices in order to turn $G$ into a graph satisfying $\mathcal{P}$.

Let $\mathcal{F}$ be a family of graphs. An {\it $\epsilon$-tester} for $\mathcal{P}$ with respect to $\mathcal{F}$ is a randomized algorithm, which, given $n$ and the ability to check whether there is an edge between a given pair of vertices, distinguishes with probability at least $2/3$ between the cases $G$ satisfies $\mathcal{P}$ and $G$ is $\epsilon$-far from satisfying $\mathcal{P}$, for every $G \in \mathcal{F}$. Such an $\epsilon$-tester is {\em one-sided} if, whenever $G \in \mathcal{F}$ satisfies $\mathcal{P}$, the $\epsilon$-tester outputs this with probability 1. A property $\mathcal{P}$ is {\it strongly testable} with respect to $\mathcal{F}$ if for every fixed $\epsilon > 0$, there exists a one-sided $\epsilon$-tester for $\mathcal{P}$ with respect to $\mathcal{F}$, whose query complexity is bounded only by a function of $\epsilon$, which is independent of the size of the input graph. The {\it vertex query complexity} of an algorithm is the number of vertices that are randomly sampled.

Property $\mathcal{P}$ is {\it easily testable} with respect to $\mathcal{F}$ if it is strongly testable with a one-sided $\epsilon$-tester whose query complexity is polynomial in $\epsilon^{-1}$, and otherwise $\mathcal{P}$ is {\it hard} with respect to $\mathcal{F}$. In classical complexity theory, an algorithm whose running time is polynomial in the input size is considered fast, and otherwise slow. This provides a nice analogue of polynomial-time algorithms for property testing. The above definitions extend to $k$-uniform hypergraphs, with pairs replaced by $k$-tuples.

A very general result of Alon and Shapira \cite{AlSh} in graph property testing states that every hereditary family $\mathcal{P}$ of graphs is strongly testable. Unfortunately, the bounds on the query complexity that this proof gives are quite enormous. They are based on the strong regularity lemma, which gives wowzer-type bounds.\footnote{Define the \emph{tower} function $T(1) = 2$ and $T(i + 1) = 2^{T(i)}$. Then the \emph{wowzer} function is defined as $W(1) = 2$ and $W(i + 1) = T(W(i))$. } Even the recently improved bound of Conlon and Fox \cite{CoFo,CoFo2} gives only a tower-type bound.\footnote{In addition to the dependence on the approximation parameter $\epsilon$, there is also a dependence on the property being tested that can make it arbitrarily hard to test \cite{AlSh}. However, such properties appear to be pathological and the standard properties that are studied should have only a weak dependence on the property being tested.}  The result of Alon and Shapira was extended to hypergraphs by R\"odl and Schacht \cite{RoSc}; see also the work of Austin and Tao \cite{AT}. These give even worse bounds, of Ackermann-type bounds.  It is known that many properties are not easily testable (see \cite{Al,AlFo,AlSh1}).

In Section~\ref{testing}, we give an application of the polynomial semi-algebraic regularity lemma, Theorem \ref{reg1}, to show how to easily test ``typical'' hereditary properties with respect to semi-algebraic graphs and hypergraphs of constant complexity. The exact statement shows that the query complexity is for every property, polynomial in a natural function of that property, which one typically expects to be polynomial. This provides another example showing that semi-algebraic graphs and hypergraphs are more nicely behaved than general ones. Using the above terminology, a special case of our results is the following.

\begin{corollary}
Let $H$ be a $k$-uniform hypergraph. Then $H$-freeness is easily testable within the family of semi-algebraic hypergraphs of bounded complexity.
\end{corollary}

For background and precise results, see Section~\ref{testing}.

\medskip

\noindent {\bf Organization.} The rest of this paper is organized as follows. In the next section, we prove Theorem \ref{main0}, giving a quantitative density theorem for $k$-uniform hypergraphs. In Section \ref{app}, we prove several geometric applications of this result. In Section \ref{structure}, we prove the quantitative regularity lemma for semi-algebraic hypergraphs. In Section \ref{testing}, we establish three results about property testing within semi-algebraic hypergraphs showing that it can be efficiently tested whether a semi-algebraic hypergraph of bounded complexity has a given hereditary property.

We systemically omit floor and ceiling signs whenever they are not crucial for the sake of clarity of presentation. We also do not make any serious attempt to optimize absolute constants in our statements and proofs.

\section{Proof of Theorem \ref{main0}}\label{semi}

 A set $A\subset \mathbb{R}^d$ is \emph{semi-algebraic} if there are polynomials $f_1,f_2,\ldots,f_t \in \mathbb{R}[x_1,\ldots,x_d]$ and a Boolean function
$\Phi$ such that

$$A =  \left\{ \textbf{x} \in \mathbb{R}^d: \Phi(f_1(\textbf{x}) \geq 0,\ldots,f_t(\textbf{x}) \geq 0)=1\right\}.$$

\noindent We say that a semi-algebraic set in $d$-space has \emph{description complexity at most $\kappa$} if the number of inequalities is at most $\kappa$, and each polynomial $f_i$ has degree at most $\kappa$.

Let us first sketch the proof of Theorem \ref{main0}.  Let $H  = (P_1\cup \cdots \cup P_k, E)$ be a $k$-partite semi-algebraic hypergraph in $\RR^d$ with complexity $(t,D)$, such that $|E| \geq \varepsilon\prod_{i = 1}^k|P_i|$.  We can assume that $D = 1$, otherwise the proof can be completed by a simple argument using the Veronese map (described below).  Set $Q = P_2\times \cdots \times P_k$, and consider the bipartite semi-algebraic graph $G = (P_1\cup Q,E)$.  Then the neighborhood of each vertex $q \in Q$ corresponds to a semi-algebraic set $q^{\ast} \subset \RR^d$ with bounded complexity, such that $pq \in E$ if and only if $p \in q^{\ast}$. We apply the well known \emph{cutting lemma} (Lemma \ref{cut} below) with parameter $ r = r(t,\varepsilon)$, to decompose $\RR^d$ into at most $O(r^d)$ relatively open simplices such that each simplex is crossed by at most $|Q|/r$ semi-algebraic sets $q^{\ast}$.  A counting argument shows that one can find a cell $\Delta$ in our decomposition such that for $P'_1 = \Delta\cap P$, we have $|P'_1| \geq \delta |P|$ where $\delta = \delta(d,k,t,\varepsilon) > 0$.  Moreover, there are at least $\gamma|Q|$ semi-algebraic sets $q^{\ast}$ that contain $\Delta$,  where $\gamma = \gamma(d,k,t,\varepsilon) > 0$.  Let $E'\subset P_2\times\cdots \times P_k$ be the set of $(k-1)$-tuples that correspond to the semi-algebraic sets $q^{\ast}\supset \Delta$.  By showing that the $(k-1)$-partite semi-algebraic hypergraph $H' = (P_2,\ldots, P_k,E')$ has bounded complexity, where $|E'| \geq \gamma \prod_{i = 2}^k|P_i|$, we can apply induction to $H'$ to find the remaining subsets $P'_2,\ldots, P'_k$.   We now flesh out the details of the proof.

Let $H = (P,E)$ be a semi-algebraic $k$-uniform hypergraph in $d$-space with complexity $(t,D)$, where $P = \{p_1,\ldots,p_n\}$.  Then there exists a semi-algebraic set

$$E^{\ast} = \left\{ (\textbf{x}_1,\ldots,\textbf{x}_k) \in \mathbb{R}^{dk}: \Phi(f_1(\textbf{x}_1,\ldots,\textbf{x}_k) \geq 0,\ldots,f_t(\textbf{x}_1,\ldots,\textbf{x}_k) \geq 0)=1\right\}$$

\noindent such that

$$(p_{i_1},\ldots,p_{i_k}) \in E^{\ast}\subset \mathbb{R}^{dk}\hspace{.4cm}  \Leftrightarrow\hspace{.4cm} (p_{i_1},\ldots,p_{i_k})  \in E.$$

\noindent  Fix $k-1$ points $q_1,\ldots,q_{k-1} \in \mathbb{R}^d$ and consider the $d$-variate polynomial $h_i(\textbf{x}) = f_i(q_1,\ldots,q_{k-1},\textbf{x})$, $1\leq i \leq t$.  We use a simple but powerful trick known as \emph{Veronese mapping} (\emph{linearization}), that transforms $h_i(\textbf{x})$ into a linear equation.  For $m = {d+  D\choose d} - 1$, we define $\phi:\mathbb{R}^d \rightarrow \mathbb{R}^m$ to be the (Veronese) map given by

$$\phi(x_1,\ldots,x_d) = (x_1^{\alpha_1}\cdots x_d^{\alpha_d})_{1 \leq \alpha_1 + \cdots + \alpha_d \leq D} \in \mathbb{R}^m.$$

\noindent Then $\phi$ maps each surface $h_i(\textbf{x}) = 0$ in $\mathbb{R}^d$ to a hyperplane $h_i^{\ast}$ in $\mathbb{R}^m$.  Note that $\phi$ is injective. For example, the $d$-variate polynomial $h_i(x_1,\ldots,x_d) = c_0 + \sum_i a_i x_1^{\alpha_1}\cdots x_d^{\alpha_d}$ would correspond to the linear equation $h_i^{\ast}(y_1,\ldots,y_m) = c_0 + \sum\limits_{i = 1}^m a_iy_i$.  Given a point set $P = \{p_1,\ldots,p_n\} \subset \mathbb{R}^d$, we write

$$\phi(P) = \{\phi(p_1),\ldots,\phi(p_n)\} \subset \mathbb{R}^m.$$

\noindent Clearly we have the following.

\begin{observation}\label{cor}
For $p \in \mathbb{R}^d$, we have $h_i(p)=h_i^{\ast}(\phi(p))$ and hence
$$\sgn(h_i(p)) = \sgn(h_i^{\ast}(\phi(p))).$$
\end{observation}

For $k\leq d$, a $k$-\emph{simplex} in $\RR^d$ is the convex hull of an affinely independent $(k+1)$-element point set.  We use the word \emph{simplex} when no assumption is made about its dimension.  Given a relatively open simplex $\Delta \subset \mathbb{R}^m$, we say that a hyperplane $h^{\ast} \subset \mathbb{R}^m$  \emph{crosses} $\Delta$ if it intersects $\Delta$ but does not contain it.  Let us recall an old lemma due to Chazelle (see also \cite{clark,chazelle}).

\begin{lemma}[\cite{Ch93}]\label{cut}
For $m\geq 1$, let $L$ be a set of $n$ hyperplanes in $\mathbb{R}^m$ and let $B\subset \RR^m$ be a bounded axis-parallel box.  Then for any integer $r$, $1 < r \leq n$, there is a subdivision of $B$ into at most $2^{10m\log (m+1)} r^m$ relatively open simplices $\Delta_i$, such that each $\Delta_i$ is crossed by at most $n/r$ hyperplanes of $L$.  Moreover, there is a deterministic algorithm for computing such a subdivision in $O(r^{m-1}n)$ time.
\end{lemma}

The main tool used in the proof of Theorem \ref{main0} is the following result on bipartite semi-algebraic graphs ($k = 2$) with point sets in different dimensions.  A very similar result was obtained by Alon et al. (see Section 6 in \cite{noga}).  Let $G = (P,Q,E)$ be a bipartite semi-algebraic graph, where $P \subset \mathbb{R}^{d_1}$, $Q \subset \mathbb{R}^{d_2}$, and $E\subset P\times Q$ has complexity $(t,D)$.  Hence, there are polynomials $f_1,f_2,\ldots,f_t$ and a Boolean formula $\Phi$ such that the semi-algebraic set

$$E^{\ast} = \{(\textbf{x}_1,\textbf{x}_2) \in \mathbb{R}^{d_1 + d_2}: \Phi(f_1(\textbf{x}_1,\textbf{x}_2) \geq 0,\ldots,f_t(\textbf{x}_1,\textbf{x}_2) \geq 0) = 1\},$$

\noindent satisfies

$$(p,q) \in E \hspace{.5cm}\Leftrightarrow\hspace{.5cm} (p,q) \in E^{\ast}\subset \mathbb{R}^{d_1 + d_2}.$$

\noindent For any point $p \in P$, the Veronese map $\phi$ maps the surface $f_i(p,\textbf{x}) = 0$ in $\mathbb{R}^{d_2}$ to a hyperplane $f^{\ast}_i(p,\textbf{y}) = 0$ in $\mathbb{R}^m$, for $1\leq i \leq t$.

\begin{lemma}\label{bialg}
Let $G = (P,Q,E)$ be as above with $|E| \geq \varepsilon |P||Q|$.  Then there are subsets $P'\subset P$ and $Q'\subset Q$ such that

$$|P'| \geq \frac{\varepsilon}{8}|P| \hspace{.5cm}\textnormal{and}\hspace{.5cm}|Q'| \geq  \frac{\varepsilon^{m + 1}}{t^{m} 2^{14m\log (m+1)}}|Q|,$$

\noindent where $m = {d_2 + D \choose d_2} - 1$, and $P'\times Q' \subset E$.  Moreover, there is a relatively open simplex $\Delta \subset \mathbb{R}^m$ such that $\phi(Q') = \phi(Q)\cap \Delta$, and for all $p \in P'$ and $1\leq i \leq t$, the hyperplane $f^{\ast}_i(p,\textbf{y}) = 0$ in $\mathbb{R}^m$ does not cross $\Delta$.

\end{lemma}

\begin{proof}

Let $S$ be the set of (at most) $t|P|$ surfaces in $\mathbb{R}^{d_2}$ defined by $f_i(p,\textbf{x}) =0$, for $1 \leq i \leq t$ and for each $p \in P$.  Set $m = {d_2 + D\choose d_2} -1$ and let $\phi:\mathbb{R}^{d_2}\rightarrow \mathbb{R}^m$ denote the Veronese map described above.  Then each surface $f_i(p,\textbf{x}) = 0$ in $\mathbb{R}^{d_2}$ will correspond to a hyperplane $f_i^{\ast}(p,\textbf{y}) = 0$ in $\mathbb{R}^m$.  Then let $L$ be the set of $t|P|$ hyperplanes in $\mathbb{R}^m$ that corresponds to the surfaces in $S$.  We set $B\subset \RR^m$ to be a sufficiently large (bounded) axis-parallel box whose interior contains $\phi(Q)$.

Let $r > 1$ be an integer that will be determined later.  By Lemma \ref{cut}, there is a subdivision of $B$ into at most $2^{10m\log(m + 1)}r^{m}$ relatively open simplices $\Delta_i$, such that each $\Delta_i$ is crossed by at most $t|P|/r$ hyperplanes from $L$.  We define $Q_i \subset Q$ such that $\phi(Q_i) = \phi(Q)\cap \Delta_i$.  If

$$|Q_i| < \frac{\varepsilon}{2\cdot 2^{10m\log (m+1)} r^{m}}|Q|,$$

\noindent then we discard all edges in $E$ that are incident to vertices in $Q_i$.  By Lemma \ref{cut}, we have removed at most

$$2^{10m\log (m+1)}r^{m}\frac{\varepsilon}{2\cdot 2^{10m\log (m+1)} r^{m}}|P||Q| = \frac{\varepsilon}{2}|P||Q|$$

\noindent edges.  Let $E_1$ be the set of remaining edges. Hence, $|E_1| \geq \varepsilon|P||Q|-(\varepsilon/2)|P||Q|=(\varepsilon/2)|P||Q|$.   For $p \in P$, if all $t$ hyperplanes $f^{\ast}_1(p,\textbf{y}) = 0,\ldots,f^{\ast}_t(p,\textbf{y}) = 0$ in $\mathbb{R}^m$ do not cross $\Delta_i$, then by Observation \ref{cor}, the sign pattern of $f_j(p,q)$ does not change over all $q \in Q_i$.  Hence, in this case, we have either $p\times Q_i\subset E_1$ or $p\times Q_i\subset \overline{E}_1$.

We define $P_i \subset P$ to be the set of points in $P$ such that $p\in P_i$ if and only if $p$ is adjacent to all points in $Q_i$ (with respect to $E_1$) and all $t$ hyperplanes $f_j(p,\textbf{x}) = 0$, $1\leq j \leq t$, in $\mathbb{R}^{m}$ do not cross the simplex $\Delta_i$.  For each fixed $i$, by Lemma \ref{cut}, we know that at most $t|P|/r$ points in $P$ gives rise to a hyperplane that crosses $\Delta_i$.

Hence there exists a $P_j$ such that $|P_j| \geq (\varepsilon/8)|P|$, since otherwise

$$|E_1| \leq \sum\limits_{i}\left(\frac{\varepsilon}{8}|P| + \frac{t|P|}{r}\right)|Q_i| = \left(\frac{\varepsilon}{8} + \frac{t}{r}\right)|P||Q|.$$

\noindent For $r = 8t/\varepsilon$, this implies $|E_1| < (\varepsilon/4)|P||Q|$, which is a contradiction.  Therefore, we have subsets $P_j \subset P$ and $Q_j \subset Q$ such that $|P_j| \geq (\varepsilon/8)|P|$ and $P_j\times Q_j \subset E_1$.  By construction of $E_1$, for $r = 8t/\varepsilon$, we have

$$|Q_j| \geq \frac{\varepsilon}{2\cdot 2^{10m\log (m+1)} r^{m}}|Q| = \frac{\varepsilon^{m + 1}}{2\cdot t^{m}2^{10m \log (m+1)}2^{3m}}|Q|\geq  \frac{\varepsilon^{m + 1}}{t^{m}2^{14m \log (m+1)}}|Q|.$$

\noindent Moreover, we obtain a simplex $\Delta_j\subset \RR^m$ such that $\phi(Q_j) = \phi(Q)\cap \Delta_j$.

\end{proof}

We are now ready to prove Theorem \ref{main0}.

\medskip

\noindent \emph{Proof of Theorem \ref{main0}.}  Let $H = (P,E)$ be a $k$-partite semi-algebraic hypergraph in $d$-space with complexity $(t,D)$, such that $P = P_1\cup \cdots \cup P_k$ and $|E| \geq \varepsilon|P_1|\cdots |P_k|$.  For $m = {d + D\choose d} - 1$, we will show that there are subsets $P'_1 \subset P_1, P'_2 \subset P_2,\ldots, P'_k \subset P_k$, such that for $1\leq i \leq k$,

$$|P'_i| \geq \frac{\varepsilon^{m + 1}}{8^{k(m + 1)} t^{m } (2m  + 3)^{km} 2^{14m \log (m+1)}}|P_i|,$$

\noindent and $P'_1\times \cdots \times P'_k \subset E$.  This would suffice to prove the statement.  We proceed by induction on $k$.  The base case $k = 2$ follows from Lemma \ref{bialg}.  Now assume that the statement holds for $k' < k$.

Since $E$ is semi-algebraic with complexity $(t,D)$, there are polynomials $f_1,\ldots,f_t$ and a Boolean formula $\Phi$ such that the semi-algebraic set

$$E^{\ast} = \{(\textbf{x}_1,\ldots,\textbf{x}_k) \in \mathbb{R}^{dk}: \Phi(f_1(\textbf{x}_1,\ldots,\textbf{x}_k) \geq 0,\ldots,f_t(\textbf{x}_1,\ldots,\textbf{x}_k) \geq 0) = 1\},$$

\noindent has the property that

$$(p_1,\ldots,p_k) \in E \hspace{.5cm}\Leftrightarrow\hspace{.5cm} (p_1,\ldots,p_k) \in E^{\ast}\subset \mathbb{R}^{dk}.$$

Set $Q = P_1\times \cdots \times P_{k-1}$, and with slight abuse of notation, let $G = (Q,P_k,E)$ be the bipartite semi-algebraic graph with relation $E$.  Just as above, we let $\phi:\mathbb{R}^{d}\rightarrow \mathbb{R}^m$ denote the Veronese map.  Then for $(p_1,\ldots,p_{k-1}) \in Q$, each surface $f_i(p_1,\ldots,p_{k-1},\textbf{x}) = 0$ in $\mathbb{R}^d$ will correspond to a hyperplane $f_i^{\ast}(p_1,\ldots,p_{k-1},\textbf{y}) = 0$ in $\mathbb{R}^m$, $1\leq i \leq t$.  By Lemma \ref{bialg}, there are subsets $Q'\subset Q$ and $P'_{k} \subset P_{k}$ such that

$$|Q'| \geq \frac{\varepsilon}{8}|Q| \hspace{.5cm}\textnormal{and}\hspace{.5cm} |P'_k| \geq \frac{\varepsilon^{m  +1}}{t^{m}2^{14m\log (m+1)}}|P_k|,$$

\noindent and $Q'\times P'_k \subset E$.  Moreover there is a relatively open simplex $\Delta\subset \mathbb{R}^m$ such that $\phi(P'_k) = \Delta\cap \phi(P_k) $, and for any $(k-1)$-tuple $(p_1,\ldots,p_{k-1}) \in Q'$, all $t$ hyperplanes $f^{\ast}_i(p_1,\ldots,p_{k-1},\textbf{y}) = 0$, $1\leq i \leq t$, do not cross $\Delta$.  Let $p^{\ast}_1,\ldots,p^{\ast}_{\ell} \in \mathbb{R}^m$ be the vertices of $\Delta$, where $\ell \leq m + 1$.

We define $H_1 = (P\setminus P_k,E_1)$ to be the $(k-1)$-partite semi-algebraic hypergraph in $d$-space on the point set $P_1\cup  \cdots \cup P_{k-1}$ with relation $E_1\subset P_1\times \cdots \times P_{k-1}$, where $(p_1,\ldots,p_{k-1}) \in E_1$ if and only if $(p_1,\ldots,p_{k-1},q) \in E$ for all $q \in P'_k$, and all $t$ hyperplanes $f^{\ast}_i(p_1,\ldots,p_{k-1},\textbf{y}) = 0$ in $\mathbb{R}^m$, $1\leq i\leq t$, do not cross $\Delta$.

Next we need to check that the complexity of $E_1$ is not too high.  In order to do this, we define relations $E_2,E_3 \subset P_1\times \cdots \times P_{k-1}$ as follows.  Fix a point $q \in P'_k$, and let

 $$E_2 = \{(p_1,\ldots,p_{k-1}) \in P_1\times \cdots \times P_{k-1}:(p_1,\ldots,p_{k-1},q) \in E\}.$$

 \noindent Since the relation $E$ has complexity $(t,D)$, the relation $E_2$ is semi-algebraic with complexity $(t,D)$.

 We define the relation $E_3\subset P_1\times \cdots \times P_{k-1}$, where $(p_1,\ldots,p_{k-1}) \in E_3$ if and only if all $t$ hyperplanes $f^{\ast}_i(p_1,\ldots,p_{k-1},\textbf{y}) = 0$, $1 \leq i \leq t$, do not cross $\Delta$ in $\mathbb{R}^m$.  Notice that a hyperplane does not cross $\Delta$ in $\mathbb{R}^m$ if and only if the vertices $p^{\ast}_1,\ldots,p^{\ast}_{\ell}$ of $\Delta$ all lie in a closed half-space defined by the hyperplane (recall that $\Delta$ is relatively open).   Therefore, for each vertex $p^{\ast}_j$ of the simplex $\Delta$, we define $t$ polynomials, each in $dk-d$ variables.

 $$g_{i,j}(\textbf{x}_1,\ldots,\textbf{x}_{k-1}) = f^{\ast}_i(\textbf{x}_1,\ldots,\textbf{x}_{k-1},p^{\ast}_j).$$

Then there is a Boolean formula $\Phi_3$ such that the semi-algebraic set $E_3^{\ast} \subset \RR^{d(k-1)}$, where $(\textbf{x}_1,\ldots,\textbf{x}_{k-1}) \in E_3^{\ast}$ if and only if

 $$\Phi_3\left(\begin{array}{ll}
                  g_{1,1}(\textbf{x}_1,\ldots,\textbf{x}_{k-1}) \geq 0, & -g_{1,1}(\textbf{x}_1,\ldots,\textbf{x}_{k-1}) \geq 0,  \ldots \\
                 g_{t,\ell}(\textbf{x}_1,\ldots,\textbf{x}_{k-1}) \geq 0, & -g_{t,\ell}(\textbf{x}_1,\ldots,\textbf{x}_{k-1}) \geq 0
               \end{array}\right) = 1,$$

 \noindent has the property that $(p_1,\ldots,p_{k-1}) \in E_3$ if and only if $(p_1,\ldots,p_{k-1}) \in E_3^{\ast} \subset \mathbb{R}^{d(k-1)}$, and $E_3$ has complexity $(2t\ell,D)$.  Notice that if $(p_1,\ldots,p_{k-1}) \in E_3$, then by Observation \ref{cor}, the sign pattern of $f_i(p_1,\ldots,p_{k-1}, q)$ does not change over all $q\in P'_k$ since $\phi(P'_k) \subset \Delta$.  Therefore, if we also have $(p_1,\ldots,p_{k-1}) \in E_2$, this implies $(p_1,\ldots,p_{k-1}) \in E_1$.  That is,

 $$(p_1,\ldots,p_{k-1}) \in E_1\hspace{.5cm}\Leftrightarrow \hspace{.5cm} (p_1,\ldots,p_{k-1}) \in E^{\ast}_1 = E_2^{\ast} \cap E_3^{\ast}\subset \mathbb{R}^{d(k-1)},$$

\noindent and $E_1$ has complexity $(t + 2t\ell,D)$.  Since $\ell \leq m + 1$, $E_1$ has complexity $(t (2m +3),D)$.  By construction of $Q'$, we have

$$|E_1| \geq |Q'| \geq \frac{\varepsilon}{8}|P_1|\cdots |P_{k-1}|.$$

\noindent By applying the induction hypothesis on the $(k-1)$-partite semi-algebraic hypergraph $H_1 = (P\setminus P_k, E_1)$, we obtain subsets $P'_i\subset P_i$, $1\leq i \leq k-1$, such that

\begin{eqnarray*}
    |P'_i| & \geq & \frac{(\varepsilon/8)^{m+1}}{8^{(k-1)(m + 1)}(t (2m + 3))^{m}(m+ 2)^{(k - 1)m}2^{14m\log (m+1)}}|P_i| \\\\
     &\geq & \frac{\varepsilon^{m+1}}{8^{k(m + 1)} t^{m} (2m + 3)^{km} 2^{14m\log (m+1)}}|P_i|,
  \end{eqnarray*}

\noindent and $P'_1\times \cdots \times P'_{k-1} \subset E_1$.  By definition of $E_1$ and by construction of $P'_k$, we have $P'_1\times \cdots \times P'_k \subset E$, and this completes the proof. $\hfill\square$

\medskip

Next, we will show that the $P'_i$-s obtained in Theorem \ref{main0} can be captured by a semi-algebraic set with bounded description complexity.  While the following variant of Theorem \ref{main0} gives a weaker lower bound for $|P'_i|$, it does however produce a semi-algebraic set $\Delta_i$ such that $P'_i = P_i\cap \Delta_i$, for $i = 1,2,\ldots, k$.

\begin{theorem}\label{sets}
For $k,d,t,D,\varepsilon > 0$, there exists a constant $\delta=\delta(k,d,t,D,\varepsilon) > 0$ with the following property. Let $H = (P_1\cup\cdots\cup P_k,E)$ be any $k$-partite semi-algebraic hypergraph in $\RR^d$ with complexity $(t,D)$ and $|E| \geq \varepsilon\Pi_{i=1}^k|P_i|$. Then one can choose subsets $P'_i \subset P_i$, $1\leq i \leq k$, such that

$$|P'_i| \geq \delta|P_i|,$$

\noindent and $P'_1\times \ldots \times P'_k \subseteq E$.  Moreover, there are semi-algebraic sets $\Delta_1,\ldots, \Delta_k$ such that $P'_i = P_i\cap \Delta_i$, and each $\Delta_i$ has complexity at most $\kappa = \kappa(k,d,t,D,\varepsilon)$.

\end{theorem}

\proof  Let $\delta = \delta(k,d,t,D,\varepsilon)$ ($\kappa = \kappa(k,d,t,D,\varepsilon$)) be sufficiently small (large) and chosen later.  Each simplex $\Delta^{\ast} \subset \RR^m$ in the proof above, corresponds to a semi-algebraic set $\Delta\subset \RR^d$.  Therefore, the proof above shows that after applying a $(k-2)$-fold application of Lemma \ref{bialg}, we obtain semi-algebraic sets $\Delta_1,\ldots, \Delta_{k-2} \subset \RR^d$, and a bipartite semi-algebraic graph $G = (P_{k-1},P_k, E')$, such that for $P'_i  = P_i\cap \Delta_i$, $1\leq i \leq k-2$, we have $(p_1,\ldots, p_k) \in E(H)$ for every $p_1\in P'_1,\ldots, p_{k-2} \in P'_{k-2},$ and $(p_{k-1},p_k) \in E'(G)$.  Moreover, we have $|P'_i| \geq \delta |P_i|$, each semi-algebraic set $\Delta_i$ has complexity at most $\kappa = \kappa(k,d,t,D,\varepsilon)$, the complexity of the edge relation $E'$ is $(c,c)$, where $c = c(\varepsilon, k,d,t,D)$, and $|E'| \geq \gamma |P_{k-1}||P_k|$ where $\gamma = \gamma(\varepsilon, k,d,t,D)$.  We obtain $P'_{k-1}$, $P'_{k}$, $\Delta_{k-1}$, and $\Delta_k$ by applying Lemma \ref{keybipartite} below to the bipartite graph $G = (P_{k-1},P_k,E')$, the statement follows and this completes the proof of Theorem \ref{sets}. \qed

\begin{lemma}\label{keybipartite}  For $c,d, \gamma > 0$, there is a $\delta = \delta(c,d,\gamma)$ and a $\kappa  = \kappa(c,d,\gamma)$ with the following property.  Let $G = (P,Q,E)$ be any bipartite semi-algebraic graph in $\RR^d$ with complexity at most $(c,c)$ and $|E|\geq \gamma |P||Q|$.  Then one can choose subsets $P'\subset P$ and $Q'\subset Q$ such that $|P'|\geq \delta|P|$, $|Q'|\geq \delta|Q|$, and $P'\times Q' \subset E$. Moreover, there are semi-algebraic sets $\Delta_{1},\Delta_2 \subset \RR^d$ such that $P' = \Delta_{1}\cap P$ and $Q' = \Delta_{2}\cap Q$, and each $\Delta_i$ has complexity at most $\kappa$.
\end{lemma}

\begin{proof}

Without loss of generality, we can assume $c > 4$.  Since $E$ is a semi-algebraic relation with complexity $(c,c)$, there are $c$ polynomials $f_1,\ldots, f_c$ and a Boolean formula $\Phi$ such that the semi-algebraic set

$$E^{\ast} = \{(\textbf{x}_1,\textbf{x}_2); \in \RR^{2d}: \Phi(f_1(\textbf{x}_1,\textbf{x}_2) \geq 0,\ldots, f_c(\textbf{x}_1,\textbf{x}_2) \geq 0) = 1\}$$

\noindent satisfies

$$(p,q) \in E \hspace{.5cm}\Leftrightarrow \hspace{.5cm} (p,q) \in E^{\ast} \subset \RR^{2d}.$$

Let $\Sigma_1$ be the set of at most $c|P|$ surfaces in $\RR^d$ defined by $f_i(p,\textbf{x}) = 0$, for $1\leq i \leq c$ and for each $p\in P$.  Likewise, let $\Sigma_2$ be the set of at most $c|Q|$ surfaces in $\RR^d$ defined by $f_i(\textbf{x},q) = 0$, for $1\leq i \leq c$ and for each $(q \in Q)$.  Then, for $m = m(c)$, let $\phi:\RR^d \rightarrow \RR^m$ denote the Veronese map such that the surfaces $f_i(p,\textbf{x}) = 0$, $f_i(\textbf{x},q) = 0$ in $\RR^d$ correspond to hyperplanes $f^{\ast}_i(p,\textbf{y}) = 0, f^{\ast}_i(\textbf{y},q) = 0$ in $\RR^m$, for $1\leq i \leq c$.  Let $\Sigma^{\ast}_1$, $\Sigma_2^{\ast}$ be the images of $\Sigma_1$, $\Sigma_2$ via $\phi$ respectively.

Let $B\subset \RR^d$ be a sufficiently large bounded axis-parallel box that contains $P\cup Q$.  Since a simplex in the image $\phi(B)\subset \RR^m$ corresponds to a semi-algebraic set in $B$ via $\phi$, we can apply Lemma~\ref{cut} with parameter $r = c^2/\gamma$ to $\Sigma^{\ast}_1$, and obtain a subdivision of the box $B$ into $s \leq c_1(c^2/\gamma)^m$ semi-algebraic sets $\Delta_1, \ldots,\Delta_s$, such that each $\Delta_i$ is crossed by at most $\gamma|P|/c$ hypersurfaces from $\Sigma_1$.  Moreover, each $\Delta_i$ has complexity at most $c_2 = c_2(c,d)$.   Likewise, given $\Sigma^{\ast}_2$, we apply Lemma \ref{cut} with parameter $r = c^2/\gamma$ to obtain another subdivision of $B$ into $s \leq c_1(c^2/\gamma)^m$ semi-algebraic sets $\Delta'_1,\ldots,\Delta'_s$, such that each $\Delta'_i$ is crossed by at most $\gamma|Q|/c$ hypersurfaces from $\Sigma_2$, and each $\Delta'_i$ has complexity at most $c_2$.

Let $U_{\ell} = Q\cap \Delta_{\ell}$ for each $\ell \leq s$.  We now partition $\Delta_{\ell}$ as follows. For $j \in \{1,\ldots, s\}$, define $\Delta_{\ell,j} \subset \Delta_{\ell}$ by

 $$\Delta_{\ell,j} = \{y \in \Delta_{\ell}: f_i(\textbf{x},y) = 0 \textnormal{ crosses } \Delta'_j \textnormal{ for some $i$}\}.$$

 \begin{observation}
For any $j$ and $\ell$, the semi-algebraic set $\Delta_{\ell,j}$ has complexity at most $c_3 = c_3(c,d)$.

 \end{observation}

 \begin{proof}

Set $\sigma(y) = \{\textbf{x} \in \RR^d: f_i(\textbf{x},y) = 0 \textnormal{ for some $i$, $1\leq i\leq c$} \}$, which is a semi-algebraic set with complexity at most $c_4 = c_4(c,d)$.  Then

$$\Delta_{\ell,j} = \left\{y \in \Delta_{\ell}: \begin{array}{l}
                                            \exists x_1 \in \RR^d \textnormal{ s.t. } x_1\in \sigma(y)\cap \Delta'_{j}, \textnormal{ and } \\
                                             \exists x_2 \in \RR^d \textnormal{ s.t. } x_2\in   \Delta'_{j}\setminus \sigma(y).
                                          \end{array}\right\}.$$

We can apply quantifier elimination (see Theorem 2.74 in \cite{basu}) to make $\Delta_{\ell,j}$ quantifier-free, with description complexity at most $c_3 = c_3(c,d)$.

 \end{proof}

Set $\mathcal{F}_{\ell} = \left\{\Delta_{\ell,j}: 1 \leq j \leq s \right\}$.  We partition the points in $\Delta_{\ell}$ into equivalence classes, where two points $u,v \in \Delta_{\ell}$ are equivalent if and only if $u$ belongs to the same members of $\mathcal{F}_{\ell}$ as $v$ does.  Since $\mathcal{F}_{\ell}$ gives rise to at most $c_3|\mathcal{F}_{\ell}|$ polynomials of degree at most $c_3$, by the Milnor-Thom theorem (see \cite{matousek} Chapter 6), the number of distinct sign patterns of these $c_3|\mathcal{F}_{\ell}|$ polynomials is at most $\left(50c_3(c_3|\mathcal{F}_{\ell}|)\right)^d.$  Hence, there is a constant $c_5 = c_5(c,d)$ such that $\Delta_{\ell}$ (and therefore $U_{\ell}$) is partitioned into at most $c_5s^d$ equivalence classes, each class being a semi-algebraic set of complexity at most $\kappa = \kappa(c,d,\gamma)$.    After repeating this procedure to each $\Delta_{\ell}$ and $U_{\ell}$, we obtain a partition of our point set $Q =  Q_1\cup \cdots \cup Q_K$ with $K = K(c,d,\gamma)$.  Moreover, we obtain semi-algebraic sets $\tilde{\Delta}_1,\ldots, \tilde{\Delta}_K$ such that $Q_i = \tilde{\Delta}_i\cap Q$, and each $\tilde{\Delta}_i$ has complexity at most $\kappa$.  We repeat this entire procedure on $P$ and obtain a partition $P = P_1\cup \cdots \cup P_K$ and semi-algebraic sets $\tilde{\Delta}'_1,\ldots, \tilde{\Delta}'_K$ such that $P_i = \tilde{\Delta}_i'\cap P$, each $\tilde{\Delta}_i'$ having complexity at most $\kappa$.

We say that a pair of parts $(P_i,Q_j)$ is \emph{bad}, if $|P_i| \leq \gamma|P|/(8K)$, or $|Q_j| \leq \gamma|Q|/(8K)$, or if $(P_i,Q_j)$ is not homogeneous.   By deleting all edges between bad pairs $(P_i,Q_j)$ such that either $|P_i| \leq \gamma|P|/(8K)$ or $|Q_j| \leq \gamma|Q|/(8K)$, we have deleted at most $\gamma|P||Q|/4$ edges.  Hence, there are at least $3\gamma|P||Q|/4$ edges remaining in $G$.

Let us now examine the bad pairs $(P_i,Q_j)$ that are not homogeneous.  Consider the part $Q_j$.  Then there is a semi-algebraic set $\Delta_{\ell}$ obtained from the earlier application of Lemma \ref{cut} such that $U_{\ell} = Q\cap \Delta_{\ell}$ and  $Q_j \subset U_{\ell} \subset \Delta_{\ell}$.  Now consider all parts $P_i$ such that $P_i$ contains a vertex $p$ that is not complete or empty to $Q_j$.  Then, by construction of our partition $P = P_1\cup \cdots\cup P_K$, each point $p\in P_i$ gives rise to a surface $\sigma \in \Sigma_1$ that crosses $\Delta_{\ell}$.  By our earlier application of Lemma \ref{cut}, the total number of such points in $P$ is at most $\gamma|P|/c$.  Therefore, by deleting all edges between such bad pairs, we have deleted at most

$$\sum\limits_i |P_i||Q_j| = |Q_j| \sum\limits_i |P_i| \leq \gamma|Q_j| |P|/c$$

\noindent edges, where the sum is over all $i$ such that $P_i$ contains a vertex $p$ that is not complete or empty to $Q_j$.  Summing over all $j$, we have deleted at most

$$\sum\limits_{i,j} |P_{i}||Q_{j}| \leq \gamma|P||Q|/c$$

\noindent edges, where the sum is taken over all pairs $i,j$ such that $P_i$ contains a vertex that is not complete or empty to $Q_j$.  A symmetric argument shows that by deleting all edges between parts $P_i$ and $Q_j$ such that $Q_j$ contains a vertex $q$ that is not complete or empty to $P_i$, we have deleted at most $\gamma|P||Q|/c$ edges.

Therefore, we have deleted in total at most $2\gamma|P||Q|/c + \gamma|P||Q|/4 < \gamma|P||Q|$ edges in our original bipartite graph $G$, which implies that there are still edges remaining.  Hence, for sufficiently small $\delta = \delta(c,d,\gamma)$ (namely $\delta = \gamma/(8K)$), there are parts $P'\subset P$ and $Q'\subset Q$ such that $P'\times Q' \subset E$, and $|P'| \geq \delta |P|$, $|Q'| \geq \delta|Q|$.  Moreover, for sufficiently large $\kappa = \kappa(c,d,\gamma)$, there are semi-algebraic sets $\tilde{\Delta}_i,\tilde{\Delta}'_j\subset \RR^d$, each having complexity at most $\kappa$, such that $P' = P\cap \tilde{\Delta}_i$ and $Q' = Q\cap \tilde{\Delta}_j$.  By renaming these semi-algebraic sets $\Delta_1$ and $\Delta_2$, this completes the proof of Lemma \ref{keybipartite}.\end{proof}

\medskip

By setting $\varepsilon = 1/2$ in Theorem \ref{sets}, we obtain the following Ramsey-type result on semi-algebraic hypergraphs.

\begin{lemma}
\label{ramsey}

Let $H = (P,E)$ be a $k$-partite semi-algebraic hypergraph in $d$-space, where $P = P_1\cup \cdots \cup P_k$, $E\subset P_1\times \cdots \times P_k$, and $E$ has complexity $(t,D)$.  Then there exist a $\delta = \delta(k,d,t,D)$ and subsets $P'_1\subset P_1,\ldots,P'_k\subset P_k$ such that for $1\leq i \leq k$,

$$|P'_i| \geq \delta |P_i|,$$

\noindent and $(P'_1,\ldots,P'_k)$ is homogeneous (i.e., either complete or empty).  Moreover, there are semi-algebraic sets $\Delta_1,\ldots,\Delta_k\subset \mathbb{R}^d$ such that $\Delta_i$ has complexity $\kappa$, where $\kappa = \kappa(k,d,t,D)$ and $P'_i = P_i \cap \Delta_i$ for all $i$.
\end{lemma}

\noindent  We note that a similar result was obtained by Fox et al.~(Theorem 8.1 in \cite{gromov}) and Bukh-Hubard (Theorem 18 in \cite{hubard}).  The proofs in both papers show that one can find $P'_1,\ldots, P'_k$ and $\Delta_1,\ldots, \Delta_{k-1}$ with the properties described in Lemma \ref{ramsey}.   However, their arguments do not show that one can also find the last semi-algebraic set $\Delta_k$.  This additional property in Lemma \ref{ramsey} will be crucial in the proof of Theorem \ref{reg1}.

In the proof of Theorem \ref{sets}, Lemma \ref{cut} is applied $k-2$ times, each time to a family of at most $tn^{k-1}$ hyperplanes with parameter at most $r = r(k,d, t, D, \varepsilon)$, to obtain subsets $P_1',\ldots, P_{k-2}'$, semi-algebraic sets $\Delta_1,\ldots, \Delta_{k-2}\subset \RR^d$, and the bipartite graph $G = (P_{k-1},P_k,E')$.  Since $k,d, t, D, \varepsilon$ are fixed constants, this can be done in $O(n^{k})$ time by Lemma \ref{cut} and the proof of Lemma \ref{bialg}.  Then the proof of Lemma~\ref{keybipartite} shows that we can find remaining subsets $P'_{k-1}$, $P'_{k}$, and semi-algebraic sets $\Delta_{k-1}$, $\Delta_k$, in $O(n^2)$ time.  Hence, we have the following algorithmic result.

\begin{theorem}\label{alg0}
Given fixed constants $k,d, t, D > 0$, there is a $\kappa = \kappa(k, d, t, D)$ and a $\delta = \delta(k, d, t, D)$ such that the following holds.  Let $H = (P_1\cup\ldots\cup P_k,E)$ be any $k$-partite semi-algebraic hypergraph in $\RR^d$ with complexity $(t,D)$ such that $|P_i| = n$ for all $i$.  Then there is a deterministic algorithm that finds semi-algebraic sets $\Delta_1,\ldots, \Delta_k$, each $\Delta_i$ having complexity at most $\kappa$, such that for $P'_i = P_i\cap \Delta_i$, we have $|P'_i| \geq \delta n$ and $(P'_1,\ldots, P'_k)$ is homogeneous.   Moreover, this algorithm runs in $O(n^{k})$ time.

\end{theorem}

\section{Applications of Corollary \ref{main2}}\label{app}  In this section, we apply Corollary 1.2 to establish better bounds for the same-type lemma (Theorem~\ref{same}) and the Tverberg-type result for simplices (Theorem \ref{pach}).

\medskip

\subsection{Same-type lemma}

\emph{Proof of Theorem \ref{same}.}  Let $P_1,\ldots,P_k$ be finite point sets in $\mathbb{R}^d$ such that $P = P_1\cup \cdots \cup P_k$ is in general position.  By a result of Goodman and Pollack (see \cite{pollack} and \cite{gp}), the number of different order-types of $k$-element point sets in $d$ dimensions is at most $k^{O(d^2k)}$.  By the pigeonhole principle, there is an order-type $\pi$ such that at least

$$k^{-O(d^2k)}|P_1|\cdots |P_k|$$

\noindent $k$-tuples $(p_1,\ldots,p_k) \in (P_1,\ldots,P_k)$ have order-type $\pi$.  We define the relation $E\subset P_1\times \cdots \times P_k$, where $(p_1,\ldots,p_k) \in E$ if and only if $(p_1,\ldots,p_k)$ has order-type $\pi$.  Next we need to check that the complexity of $E$ is not too high.

We can check to see if $(p_1,\ldots,p_k)$ has order-type $\pi$ by simply checking the orientation of every $(d+1)$-tuple of $(p_1,\ldots,p_k)$.  More specifically, for $\textbf{x}_i = (x_{i,1},\ldots,x_{i,d})$, we define the $(d^2+d)$-variate polynomial

$$f(\textbf{x}_1,\ldots,\textbf{x}_{d+1}) = \det\left(\begin{array}{cccc}
                                                           1 & 1 & \cdots & 1 \\
                                                           x_{1,1} & x_{2,1} & \cdots & x_{d+1,1} \\
                                                           \vdots & \vdots & \vdots & \vdots \\
                                                           x_{1,d} & x_{2,d} & \cdots & x_{d+1,d}
                                                         \end{array}\right).$$

\noindent Then there exists a Boolean formula $\Phi$, such that the semi-algebraic set

$$E^{\ast} = \{(\textbf{x}_1,\ldots,\textbf{x}_k) \in \mathbb{R}^{dk}: \Phi(  \{f(\textbf{x}_{i_1},\ldots,\textbf{x}_{i_{d+1}}) \geq 0\}_{1 \leq i_1 < \cdots < i_{d+1} \leq k}) = 1\},$$

\noindent which is defined by the ${k\choose d+1}$ polynomials $f(\textbf{x}_{i_1},\ldots,\textbf{x}_{i_{d+1}})$, $1\leq i_1 < \cdots < i_{d+1}\leq k$, has the property that

$$(p_1,\ldots,p_k) \in E \hspace{.5cm}\Leftrightarrow \hspace{.5cm} (p_1,\ldots,p_k) \in E^{\ast}.$$

\noindent  Hence, the complexity of $E$ is $\left({k \choose d+1},1\right)$.  Therefore, we can apply Corollary \ref{main2} to the $k$-partite semi-algebraic hypergraph $H = (P,E)$ with $\varepsilon = k^{-O(d^2k)}$ and $t = {k\choose d+1}$ to obtain subsets $P'_1\subset P_1,\ldots,P'_k\subset P_k$ such that $(P'_1,\ldots,P'_k)$ has same-type transversals (each transversal has order-type $\pi$), and

$$|P'_i| \geq 2^{-O(d^3k\log k)}|P_i|,$$

\noindent for $1 \leq i \leq k$. $\hfill\square$

\subsection{A Tverberg-type result}

The proof of Theorem \ref{pach} requires the following result of Karasev \cite{kar} (see also Theorem 4 in \cite{kyncl}).

\begin{lemma}[\cite{kar}]\label{karasev}
Let $P_1,\ldots, P_{d+1} \subset \RR^d$ be disjoint $n$-element point sets with $P_1\cup \cdots \cup P_{d+1}$ in general position.  Then there is a point $q \in \RR^d$ which is contained in the interior of at least

$$\frac{1}{(d+1)!}n^{d + 1}$$

\noindent rainbow simplices, where a rainbow simplex is a simplex generated by selecting one point from each $P_i$.

\end{lemma}

\medskip

\noindent \emph{Proof of Theorem \ref{pach}.} We may assume $n= 2^{\Omega(d^2 \log d)}$ as otherwise we can take $p_i \in P_i$, $P_i'=\{p_i\}$, and $q$ to be any point in the simplex with vertices $p_1,\ldots,p_{d+1}$.

 Let $P_1,\ldots,P_{d+1}$ be disjoint $n$-element point sets with $P_1\cup \cdots \cup P_{d+1}$ in general position. If a simplex has one vertex in each $P_i$, then we call it \emph{rainbow}.  Hence the number of rainbow simplices is $N = n^{d+1}$.

By Lemma \ref{karasev}, there is a point $q$ contained in the interior of at least

$$\frac{1}{(d + 1)!}n^{d+1}$$

\noindent rainbow simplices.

We define $H = (P,E)$ to be the $(d+1)$-partite semi-algebraic hypergraph, where $P = P_1\cup \cdots \cup P_{d+1}$, $E\subset P_1\times \cdots \times P_{d+1}$, where $(p_1,\ldots,p_{d+1}) \in E$ if and only if $q \in \conv(p_1\cup \cdots \cup p_{d+1})$.  Next we need to check that the complexity of $E$ is not too high.

To see if $q \in \conv(p_1\cup \cdots \cup p_{d+1})$, we just need to check that the points $q$ and $p_j$ lie on the same side of the hyperplane spanned by the $d$-tuple $(p_1,\ldots,p_{j-1},p_{j+1},\ldots,p_{d+1})$, for each $1\leq j \leq d+1$.  More specifically, for $\textbf{x}_i = (x_{i,1},\ldots,x_{i,d})$, we define the $(d^2 + d)$-variate polynomial

$$f(\textbf{x}_1,\ldots,\textbf{x}_{d+1}) =  \det\left(\begin{array}{cccc}
                                                           1 & 1 & \cdots & 1 \\
                                                           x_{1,1} & x_{2,1} & \cdots & x_{d+1,1} \\
                                                           \vdots & \vdots & \vdots & \vdots \\
                                                           x_{1,d} & x_{2,d} & \cdots & x_{d+1,d}
                                                         \end{array}\right).$$

\noindent Then there exists a Boolean formula $\Phi$ such that the semi-algebraic set

$$E^{\ast} = \left\{(\textbf{x}_1,\ldots,\textbf{x}_{d+1}) \in \mathbb{R}^{d(d+1)}: \Phi\left( \left\{  \begin{array}{c}
                                                                                             f(\textbf{x}_1,\ldots,\textbf{x}_{j-1},\textbf{x}_{j+1},\ldots,\textbf{x}_{d+1},q) \geq 0, \\
                                                                                              f(\textbf{x}_1,\ldots,\textbf{x}_{j-1},\textbf{x}_{j+1},\ldots,\textbf{x}_{d+1},\textbf{x}_j) \geq 0
                                                                                           \end{array}\right\}_{1\leq j \leq d+1}\right) = 1\right\},$$

\noindent satisfies

$$(p_1,..,p_{d+1}) \in E \hspace{.5cm}\Leftrightarrow\hspace{.5cm} (p_1,\ldots,p_{d+1}) \in E^{\ast}.$$

\noindent Hence $E$ has complexity $(2(d+1),1)$, and

$$|E| \geq \frac{1}{(d + 1)!}|P_1|\cdots |P_{d+1}|.$$

\noindent Thus, we can apply Corollary \ref{main2} to $H = (P,E)$ with $\varepsilon = 1/(d+1)!$ and $t = 2(d+1)$, to obtain subsets $P'_1\subset P_1,\ldots,P'_{d+1} \subset P_{d+1}$ such that

$$|P'_i| \geq 2^{-O(d^2 \log(d+1))}|P_i|,$$

\noindent and all closed simplices with one vertex from each $P'_i$ contains $q$. $\hfill\square$

\section{Regularity lemma for semi-algebraic hypergraphs}
\label{structure}

In this section, we prove Theorem \ref{reg1}.  We first prove the following variant of Theorem~\ref{reg1}.

\begin{theorem}
\label{reg}
For any $\varepsilon > 0$, the vertex set of any semi-algebraic $k$-uniform hypergraph $H = (P,E)$ in $d$-space with complexity $(t,D)$, can be partitioned into $K\leq (1/\varepsilon)^{c}$ parts $P = P_1\cup \cdots \cup P_K$, where $c = c(k,d,t,D)$ is a constant, such that

$$\sum\frac{|P_{j_1}|\cdots |P_{j_k}|}{|P|^k} \leq \varepsilon.$$

\noindent Here the sum is taken over all $k$-tuples $(j_1,\ldots,j_k)$ such that $(P_{j_1},\ldots,P_{j_k})$ is not homogeneous.
\end{theorem}
\begin{proof}
 Let $\varepsilon > 0$ and $H = (P,E)$ be an $n$-vertex $k$-uniform semi-algebraic hypergraph with complexity $(t,D)$ in $d$-space.  Recall that $E$ is a symmetric relation. For every integer $r \geq 0$, we will recursively define a partition $\mathcal{P}_r$ on $P^k = P\times \cdots \times P$ into at most $2^{kr}$ parts of the form $X_1\times \cdots \times X_k$, such that at most $(1-\delta^k)^r|P|^k$ $k$-tuples $(p_1,\ldots,p_k)$ lie in a part $X_1\times \cdots \times X_k$ with the property that $(X_1,\ldots,X_k)$ is not homogeneous.  Note that $\delta$ is defined in Lemma \ref{ramsey}.  Also, for each $r\geq 0$, we will inductively define a collection $\mathcal{F}_r$ of semi-algebraic sets, each set with complexity at most $\kappa = \kappa(k,d,t,D)$, such that $|\mathcal{F}_r| \leq \sum_{j=0}^{r} 2^{kj}$, and for each part $X_1\times \cdots \times X_k$ in $\mathcal{P}_r$, there are subcollections $\mathcal{S}_1,\ldots,\mathcal{S}_k \subset \mathcal{F}_r$ such that

    $$X_i = \left(\bigcap\limits_{\Delta \in \mathcal{S}_i}\Delta\right)\cap P.$$

\noindent Here $\kappa = \kappa(k,d,t,D)$ is the same constant as in Lemma \ref{ramsey}.  Given such a partition $\mathcal{P}_r$, a $k$-tuple $(p_1,\ldots,p_k) \in P\times\cdots \times P$ is called \emph{bad} if $(p_1,\ldots,p_k)$ lies in a part $X_1\times \cdots \times X_k$ in $\mathcal{P}_r$, for which $(X_1,\ldots,X_k)$ is not homogeneous.

We start with $\mathcal{P}_0 = \{P\times \cdots \times P\}$ and $\mathcal{F}_0 = \{\mathbb{R}^d\}$, which satisfies the base case $r=0$.  After obtaining $\mathcal{P}_i$ and $\mathcal{F}_i$, we define $\mathcal{P}_{i + 1}$ and $\mathcal{F}_{i+1}$ as follows.  Let $X_1\times\cdots \times X_k$ be a part in the partition $\mathcal{P}_i$.  Then if $(X_1,\ldots,X_k)$ is homogeneous, we put $X_1\times\cdots \times X_k$ in $\mathcal{P}_{i+1}$.  If $(X_1,\ldots,X_k)$ is not homogeneous, then notice that $(X_1,\ldots,X_k)$ gives rise to $|X_1|\cdots|X_k|$ bad $k$-tuples $(p_1,\ldots,p_k)$.  Hence, we apply Lemma \ref{ramsey} on $(X_1,\ldots,X_k)$ to obtain subsets $X_1'\subset X_1$,\ldots,$X_k'\subset X_k$ and semi-algebraic sets $\Delta_1,....,\Delta_k$ with the properties described above.  Then we partition $X_1\times\cdots \times X_k$ into $2^k$ parts $Z_1 \times \cdots \times Z_k$ where $Z_i \in \{X_i',X_i \setminus X_i'\}$ for $1 \leq i \leq k$, and put these parts into $\mathcal{P}_{i + 1}$.  The collection $\mathcal{F}_{i+1}$ will consist of all semi-algebraic sets from $\mathcal{F}_i$, plus all semi-algebraic sets $\Delta_1,\ldots,\Delta_k,\mathbb{R}^d\setminus \Delta_1,\ldots,\mathbb{R}^d\setminus \Delta_k$ that were obtained after applying Lemma \ref{ramsey} to each part $X_1\times \cdots \times X_k$ in $\mathcal{P}_r$ for which $(X_1,\ldots,X_k)$ is not homogeneous.

By the induction hypothesis, the number of bad $k$-tuples $(p_1,\ldots,p_k)$ in $\mathcal{P}_{i + 1}$ is at most

  $$(1 - \delta^k) (1-\delta^k)^i|P|^k =(1-\delta^k)^{i+ 1}|P|^k.$$

\noindent The number of parts in $\mathcal{P}_{i + 1}$ is at most $2^k\cdot 2^{ki} = 2^{k(i + 1)}$, and $|\mathcal{F}_{i+1}| \leq |\mathcal{F}_i| + 2k|\mathcal{P}_i| \leq |\mathcal{F}_i|+2^{k(i+1)} \leq \sum_{j = 0}^{i+1} 2^{kj}$.  By the induction hypothesis, for any part $X_1\times\cdots \times X_k$ in $\mathcal{P}_{i+1}$ such that $X_1\times\cdots \times X_k$ was also in $\mathcal{P}_{i}$, there are subcollections $\mathcal{S}_1,\ldots,\mathcal{S}_k \subset \mathcal{F}_i$ such that

    $$X_i = \left(\bigcap\limits_{\Delta \in \mathcal{S}_i}\Delta\right)\cap P,$$

\noindent for $1\leq i \leq k$.  If $X_1\times \cdots \times X_k$ is not in $\mathcal{P}_i$, then there must be a part $Y_1\times\cdots \times Y_k$ in $\mathcal{P}_i$, such that $X_1\times\cdots \times X_k$ is one of the $2^k$ parts obtained from applying Lemma \ref{ramsey} to $Y_1\times\cdots \times Y_k$.  Hence, $X_i\subset Y_i$ for $1 \leq i \leq k$.  Let $\Delta_1,\ldots,\Delta_k$ be the semi-algebraic sets obtained when applying Lemma \ref{ramsey} to $Y_1\times\cdots \times Y_k$.  By the induction hypothesis, we know that there are subcollections  $\mathcal{S}_1,\ldots,\mathcal{S}_k \subset \mathcal{F}_i$ such that

    $$Y_i = \left(\bigcap\limits_{\Delta \in \mathcal{S}_i}\Delta\right)\cap P,$$

\noindent for $1\leq i \leq k$.  Hence, there are subcollections $\mathcal{S}'_i \subset \mathcal{S}_i\cup \{\Delta_i,\mathbb{R}^d\setminus\Delta_i\}$ such that

   $$X_i = \left(\bigcap\limits_{\Delta \in \mathcal{S}_i'}\Delta\right)\cap P,$$

\noindent for $1\leq i \leq k$.  We have therefore obtained our desired partition $\mathcal{P}_{i+1}$ on $P\times\cdots \times P$, and collection $\mathcal{F}_{i+1}$ of semi-algebraic sets.

At step $r =\frac{\log \varepsilon}{\log(1- \delta^k)}$, there are at most $\varepsilon |P|^k$ bad $k$-tuples $(p_1,\ldots,p_k)$ in partition $\mathcal{P}_r$.  The number of parts of $\mathcal{P}_r$ is at most $(1/\varepsilon)^{c_1}$ and $|\mathcal{F}_r| \leq (1/\varepsilon)^{c_2}$, where $c_1 = c_1(k,d,t,D)$ and $c_2 = c_2(k,d,t,D)$ (recall that $\delta = \delta(k,d,t,D)$).

Finally, we partition the vertex set $P$ into $K$ parts, $P_1,P_2,\ldots,P_K$, such that two vertices are in the same part if and only if every member of $\mathcal{F}_r$ contains both or neither of them.  Since $\mathcal{F}_r$ consists of at most $(1/\varepsilon)^{c_2}$ semi-algebraic sets, and each set has complexity at most $c=c(k,d,t,D)$, we have $K\leq (1/\varepsilon)^{c_3}$ where $c_3 = c_3(k,d,t,D)$ (see Theorem 6.2.1 in \cite{matousek}).  Now we just need to show that

$$\sum|P_{j_1}|\cdots|P_{j_k}| < \varepsilon |P|^k,$$

\noindent where the sum is taken over all $k$-tuples $(j_1,\ldots,j_k)$, $1\leq j_1 < \cdots < j_k \leq K$, such that $(P_{j_1},\ldots,P_{j_k})$ is not homogeneous.  It suffices to show that for a $k$-tuple $(j_1,\ldots,j_k)$, if $(P_{j_1},\ldots,P_{j_k})$ is not homogeneous, then all $k$-tuples $(p_1,\ldots,p_k) \in P_{j_1}\times \cdots \times P_{j_k}$ are bad in the partition $\mathcal{P}_r$.

For the sake of contradiction, suppose that $(P_{j_1},\ldots,P_{j_k})$ is not homogeneous, and that the $k$-tuple $(p_1,\ldots,p_k) \in P_{j_1}\times \cdots \times P_{j_k}$ is not bad. Then there is a part $X_1\times \cdots \times X_k$ in the partition $\mathcal{P}_r$ such that $p_i \in X_i$ for all $i$, and $(X_1,\ldots,X_k)$ is homogeneous.  Hence, there are subcollections $\mathcal{S}_1,\ldots,\mathcal{S}_k \subset \mathcal{F}_r$ such that

   $$X_i = \left(\bigcap\limits_{\Delta \in \mathcal{S}_i}\Delta\right)\cap P,$$

\noindent for $1\leq i \leq k$.  However, by construction of $P_{j_1},\ldots,P_{j_k}$, this implies that $P_{j_i}\subset X_i$ for all $i$, and we have a contradiction.  Therefore, we have obtained our desired partition $P_1,P_2,\ldots,P_K$.
\end{proof}

Although Theorem \ref{reg} does not necessarily give an equitable partition of $P$, Theorem \ref{reg1} now quickly follows.

\medskip

\noindent \emph{Proof of Theorem \ref{reg1}.} Apply Theorem \ref{reg} with approximation parameter $\varepsilon/2$. So there is a partition $\mathcal{Q}:P=Q_1 \cup \cdots \cup Q_{K'}$ into $K' \leq (2/\varepsilon)^c$ parts, where $c=c(k,d,t,D)$, such that $\sum |Q_{i_1}||Q_{i_2}|\cdots|Q_{i_k}| \leq (\varepsilon/2) |P|^k$, where the sum is taken over all $k$-tuples $(i_1,\ldots,i_k)$ such that $(Q_{i_1},\ldots,Q_{i_k})$ is not homogeneous.

Let $K=4k\varepsilon^{-1}K'$. Partition each part $Q_i$ into parts of size $|P|/K$ and possibly one additional part of size less than $|P|/K$. Collect these additional parts and divide them into parts of size $|P|/K$ to obtain an equitable partition $\mathcal{P}:P=P_1 \cup \cdots \cup P_K$ into $K$ parts. The number of vertices of $P$ which are in parts $P_i$ that are not contained in a part of $\mathcal{Q}$ is at most $K'|P|/K$. Hence, the fraction of $k$-tuples $P_{i_1} \times \cdots \times P_{i_k}$ with not all $P_{i_1},\ldots,P_{i_k}$ subsets of parts of $\mathcal{Q}$ is at most $kK'/K=\varepsilon/4$. As $\varepsilon/2+\varepsilon/4<\varepsilon$, we obtain that less than an $\varepsilon$-fraction of the $k$-tuples of parts of $\mathcal{P}$ are not homogeneous, which completes the proof.
$\hfill\square$

\medskip

The proof of Theorem \ref{reg} shows that we can obtain such a partition of $P$ in $O(\varepsilon^{-c}n^{k-1})$ time, where $c = c(k,d,t,D)$.  Indeed, we apply Theorem \ref{alg0} $\varepsilon^{-c}$ times to obtain the family of semi-algebraic sets $\mathcal{F}_r$, where $|\mathcal{F}_r| = O(\varepsilon^{-c})$.  This can be done in  $O(\varepsilon^{-c}n^{k-1})$ time.   We then partition our point set $P$ by checking which sets of $\mathcal{F}_r$ the points of $P$ lie in.  This can be done in $O(\varepsilon^{-c}n)$ time.  Finally, the argument above shows that we can refine our partition to obtain an equitable partition of $P$ satisfying the properties of Theorem \ref{reg1}.  This refinement can be done in $O(n)$ time.   This gives us the following algorithmic result.

\begin{corollary}\label{regalg}
For fixed constants $k,d,t,D > 0$, let $0 < \varepsilon < 1/2$ and $H = (P,E)$ be a $k$-uniform semi-algebraic hypergraph in $\RR^d$ with complexity $(t,D)$.  Then there is a deterministic algorithm that finds a partition of $P$ satisfying the properties in Theorem \ref{reg1} that runs in $O(\varepsilon^{-c}n^{k-1})$ time, where $c = (k,d,t, D)$.

\end{corollary}

Let us remark that for fixed $\varepsilon > 0$ and for $k = 2$, the algorithm above runs in $O(n)$ time which is best possible.  Moreover, this is much faster than the best known deterministic algorithm for Szemer\'edi's regularity lemma for graphs, which runs in $O(n^2)$ time and cannot be improved \cite{KRT}.

By copying the proof of Theorem \ref{reg1} almost verbatim, using the same-type lemma of B\'ar\'any and Valtr in \cite{barany} instead of Lemma \ref{ramsey}, we have the following.

\begin{theorem}

For any integers $d,k \geq 1$, there is a $C = C(d,k)$ such that the following holds.  For each $0 < \varepsilon < 1/2$ and for any finite point set $P$ in $\RR^d$, there is an equitable partition $P = P_1\cup P_2\cup \cdots \cup P_K$, with $K$ at most $\varepsilon^{-C}$, such that all but at most $\varepsilon{K\choose k}$ $k$-tuple of parts $(P_{i_1},\ldots,P_{i_k})$ have same-type transversals.

\end{theorem}

\section{Property testing in semi-algebraic hypergraphs}\label{testing}

In this section, we apply the polynomial semi-algebraic regularity lemma, Theorem \ref{reg1}, to quickly distinguish between semi-algebraic objects that satisfy a property from objects that are far from satisfying it. In the first subsection, we restrict ourselves to testing {\em monotone} hypergraph properties. We then discuss and prove a result about easily testing {\em hereditary} properties of graphs. We conclude with a result on easily testing hypergraph hereditary properties.  All semi-algebraic hypergraphs we consider in this section are assumed to be $k$-uniform and equipped with a symmetric relation $E$.

\subsection{Testing monotone properties}

Let $\mathcal{P}$ be a monotone property of hypergraphs, and $\mathcal{H}$ be the family of minimal forbidden hypergraphs for $\mathcal{P}$. That is, $H \in \mathcal{H}$ if $H \not \in \mathcal{P}$, but every proper subhypergraph of $H$ is in $\mathcal{P}$. We say that a hypergraph $H$ has a {\em homomorphism} to another hypergraph $R$, and write $H \rightarrow R$, if there is a mapping $f:V(H) \rightarrow V(R)$ such that the image of every edge of $H$ is an edge of $R$.

We let $\mathcal{H}_r$ denote the family of hypergraphs $R$ on at most $r$ vertices for which there is a hypergraph $H \in \mathcal{H}$ with $H \rightarrow R$. Define $$\Psi_1(\mathcal{H},r)=\max_{R \in \mathcal{H}_r}\min_{H \in \mathcal{H},H \rightarrow R} |V(H)|.$$

The following result implies that we can easily test every monotone property $\mathcal{P}$ whose corresponding function $\Psi_1(\mathcal{H},r)$ grows at most polynomially in $r$. A simple example in which $\Psi_1(\mathcal{H},r)$ is constant is the case that the property $\mathcal{P}$ is $H$-freeness for a fixed hypergraph $H$, i.e., $\mathcal{P}$ is the family of $k$-uniform hypergraphs which do not contain $H$ as a subhypergraph.

\begin{theorem}\label{monotonetest}
Let $\mathcal{P}$ be a monotone property of hypergraphs, and $\mathcal{H}$ be the family of minimal forbidden hypergraphs for $\mathcal{P}$. Within the family $\mathcal{A}$ of semi-algebraic hypergraphs in $d$-space with description complexity $(t,D)$, the property $\mathcal{P}$ can be $\epsilon$-tested with vertex query complexity at most $8(r\Psi_1(\mathcal{H},r))^2$, where $r=(1/\epsilon)^c$ with $c=c(k,d,t,D)$ is the number of parts in the algebraic regularity lemma for hypergraphs as in Theorem \ref{reg1}.
\end{theorem}
\begin{proof}
Let $s=\Psi_1(\mathcal{H},r)$ and $v=8(rs)^2$. Consider the tester which samples $v$ vertices from a hypergraph $A \in \mathcal{A}$. It accepts if the induced subhypergraph on these $v$ vertices has property $\mathcal{P}$ and rejects otherwise. It suffices to show that if $A \in \mathcal{A}$ is $\epsilon$-far from satisfying $\mathcal{P}$, then with probability at least $2/3$, the tester will reject.

Consider an equitable partition $\mathcal{Q}:V(A)=V_1 \cup \ldots \cup V_{r'}$ of the vertex set of $A$ guaranteed by Theorem \ref{reg1} such that all but at most an $\epsilon$-fraction of the $k$-tuples of parts are homogeneous. Let $r'=|\mathcal{Q}|$ be the number of parts of the partition, so $r' \leq r$. Delete all edges of $A$ whose vertices go between parts which are not complete. By the almost homogeneous property of the partition, at most an $\epsilon$-fraction of the edges are deleted. Let $A'$ denote the resulting subhypergraph of $A$. As $\mathcal{P}$ is monotone, if $A \in \mathcal{P}$, then $A' \in \mathcal{P}$.  Let $R$ be the hypergraph on $[r']$, which has one vertex for each part in $\mathcal{Q}$, and a $k$-tuple $(i_1,\ldots,i_k)$ of (not necessarily distinct) vertices of $R$ forms an edge if and only if the corresponding $k$-tuple $V_{i_1},\ldots,V_{i_k}$ of parts are complete in $A'$ (and hence in $A$ as well). If $R \not \in \mathcal{H}_r$, then every hypergraph which has a homomorphism to $R$ is in $\mathcal{P}$ and hence $A' \in \mathcal{P}$. However, at most an $\epsilon$-fraction of the $k$-tuples are deleted from $A$ to obtain $A'$, and so $A$ is not $\epsilon$-far from satisfying $\mathcal{P}$, contradicting the assumption. Hence, $R \in \mathcal{H}_r$.  Since $R \in \mathcal{H}_r$, there is a hypergraph $H \in \mathcal{H}$ on at most $s=\Psi_1(\mathcal{H},r)$ vertices with $H \rightarrow R$. Consider such a homomorphism $f:V(H) \rightarrow V(R)$, and let $a_i=|f^{-1}(i)|$. If the sampled $v$ vertices contain at least $a_i$ vertices in $V_i$ for each $i$, then $H$ is a subgraph of the sampled vertices and hence the sampled vertices do not have property $\mathcal{P}$.

So we need to estimate the probability of the event that the sampled $v$ vertices contain $a_i$ vertices in $V_i$ for each $i$. For a particular $i$, the probability that the sampled $v$ vertices contain fewer than $a=a_i$ vertices in $V_i$ is $0$ if $a=0$ and is otherwise less then \begin{eqnarray*} (1-1/r')^{v-a}{v \choose a} < e^{-v/(2r')}v^a  < 1/(4s),\end{eqnarray*}
where we used the union bound and that apart from $a$ of the $v$ vertices being allowed to be in $V_i$, the other $v-a$ coordinates are not allowed to be in $V_i$, which has order $\frac{1}{r'}|V(A)|$.  In the last inequality we use that $a \leq s$.
Taking the union bound and summing over all $i$, the probability that the sampled set of $v$ vertices does not contain $a_i$ vertices in $V_i$ for at least one $i$ is at most $s \times 1/(4s)=1/4$. This completes the proof.
\end{proof}

\subsection{Testing hereditary properties of graphs}

We next state and prove a result which shows that typical hereditary properties of graphs are easily testable within semi-algebraic graphs.  We say that for a graph $H$ and a graph $R$ on $[r']$ with loops, there is an {\it induced homomorphism} from $H$ to $R$, and we write $H \rightarrow_{\textrm{ind}} R$, if there is a mapping $f:V(H) \rightarrow V(R)$ which maps edges of $H$ to edges of $R$, and every nonadjacent pair of distinct vertices of $H$ gets mapped to a nonadjacent pair in $R$. We write $H \not \rightarrow_{\textrm{ind}} R$ if $H \rightarrow_{\textrm{ind}} R$ does not hold.

Let $P = P_1 \cup \ldots \cup P_{r'}$ be a vertex partition of a semi-algebraic graph $G$. A key observation is that if we \emph{round} $G$ by the partition of $P$ and the graph $R$ with loops to obtain a graph $G'$ on the same vertex set as $G$ by adding edges to make $P_i, P_j$ complete if $(i,j)$ is an edge of $R$, and deleting edges to make $P_i,P_j$ empty if $(i,j)$ is not an edge of $R$ and we have that $H \not \rightarrow_{\textrm{ind}} R$, then $G'$ does not contain $H$ as an induced subgraph.

Let $\mathcal{P}$ be a hereditary graph property, and $\mathcal{H}$ be the family of minimal (induced) forbidden graphs for $\mathcal{P}$. That is, each $H \in \mathcal{H}$ satisfies $H \not \in \mathcal{P}$, but every proper induced subgraph $H'$ of $H$ satisfies $H' \in \mathcal{P}$. For a nonnegative integer $r$,
let $\mathcal{H}_r$ be the family of graphs $R$ on at most $r$ vertices for which there is at least one $H \in \mathcal{H}$ such that $H \rightarrow_{\textrm{ind}} R$.  As long as $\mathcal{H}_r$ is nonempty, define
$$\Psi_2(\mathcal{H},r) = \max_{R \in \mathcal{H}_r} \min_{H \in \mathcal{H}:H\rightarrow_{\textrm{ind}} R} |V(H)|.$$ If $\mathcal{H}_r$ is empty, then we define $\Psi_2(\mathcal{H},r) =1$. Note that $\Psi_2(\mathcal{H},r)$ is a monotonically increasing function of $r$.

We now state our main result for testing hereditary graph properties within semi-algebraic graphs. It implies that, if $\Psi_2(\mathcal{H},r)$ is at most polynomial in $r$, then $\mathcal{P}$ can be easily tested within the family of semi-algebraic graphs of constant description complexity, i.e., there is an $\epsilon$-tester with vertex query complexity $\epsilon^{-O(1)}$. A simple example for which $\Psi_2(\mathcal{H},r)$ is constant is the case that the property $\mathcal{P}$ is the family of graphs which do not contain an induced subgraph isomorphic to $H$, for some fixed graph $H$.

\begin{theorem}
\label{graphprop}
Let $\mathcal{P}$ be a hereditary property of graphs, and $\mathcal{H}$ be the family of minimal forbidden graphs for $\mathcal{P}$. Within the family $\mathcal{A}$ of semi-algebraic graphs in $d$-space with description complexity $(t,D)$, property $\mathcal{P}$ can be $\epsilon$-tested with vertex query complexity at most $(r\Psi_2(\mathcal{H},r))^C$, where $r=(1/\epsilon)^C$ with $C=C(d,t,D)$.
\end{theorem}

We show how this theorem can be established using the following ``strong regularity lemma'' for semi-algebraic graphs.

\begin{theorem}
\label{reg7}
For any $0<\alpha,\epsilon < 1/2$, any semi-algebraic graph $H = (P,E)$ in $d$-space with complexity $(t,D)$ has an equitable vertex partition $P = P_1\cup \cdots \cup P_{r'}$ with $r' \leq r=(1/\epsilon)^{c'}$ with $c'=c'(d,t,D)$ such that all but an $\epsilon$-fraction of the pairs $P_i,P_j$ are homogeneous. Furthermore, there are subsets $Q_i \subset P_i$ such that each pair $Q_i,Q_j$ with $i \not = j$ is complete or empty, and each $Q_i$ has density at most $\alpha$ or at least $1-\alpha$. Moreover, $|Q_i| \geq \delta |P|$ with $\delta=(\alpha\epsilon)^{c}$ with $c=c(d,t,D)$.
\end{theorem}

We next prove Theorem \ref{graphprop} assuming Theorem \ref{reg7}. The rest of the subsection is then devoted to proving Theorem \ref{reg7}.
\smallskip

\noindent {\it Proof of Theorem \ref{graphprop}.} Let $s=\Psi_2(\mathcal{H},r)$ and $v=(rs)^C$ for an appropriate constant $C=C(d,t,D)$. Consider the tester which samples $v$ vertices from a graph $A=(P,E) \in \mathcal{A}$. It accepts if the induced subgraph on these $v$ vertices has property $\mathcal{P}$ and rejects otherwise. It suffices to show that if $A \in \mathcal{A}$ is $\epsilon$-far from satisfying $\mathcal{P}$, then with probability at least $2/3$, the tester will reject.

Consider an equitable partition $P=P_1 \cup \ldots \cup P_{r'}$ with $r' \leq r=(1/\epsilon)^{c'}$ of the vertex set of $A$ guaranteed by Theorem \ref{reg7} with the property that all but at most an $\epsilon$-fraction of the pairs of parts are homogeneous, and, with $\alpha=1/(4s^2)$, there are subsets $Q_i \subset P_i$ such that each pair $Q_i,Q_j$ with $i \not = j$ is complete or empty, each $Q_i$ has density at most $\alpha$ or at least $1-\alpha$, and $|Q_i| \geq \delta |P|$ with $\delta=(\alpha\epsilon)^{c}$, where $c=c(d,t,D)$. Let $R$ be the graph on $[r']$ with loops where $(i,j)$ is an edge of $R$ if and only if $Q_i,Q_j$ is complete to each other if $i \not = j$, and $(i,i)$ is an edge if the density in $Q_i$ is at least $1-\alpha$. Round $A$ by the partition of $P$ and the graph $R$ to obtain another graph $A'$. That is, $A'$ has the same vertex set as $A$, and we delete the edges between $P_i$ and $P_j$  if $(i,j)$ is not an edge of $R$ and add all possible edges between $P_i$ and $P_j$ if $(i,j)$ is an edge of $R$. The resulting graph $A'$ is homogeneous between every pair of parts and at most an $\epsilon$-fraction of the pairs of vertices were added or deleted as edges from $A$ to obtain $A'$.  This is because only an $\epsilon$-fraction of the pairs of parts of the partition of $P$ are not homogeneous, and only edges between pairs of nonhomogeneous pairs are added or deleted.

If $R \not \in \mathcal{H}_r$, then every graph which has an induced homomorphism to $R$ is in $\mathcal{P}$ and hence $A' \in \mathcal{P}$. However, at most an $\epsilon$-fraction of the pairs were added or deleted from $A$ to obtain $A'$, and so $A$ is not $\epsilon$-far from satisfying $\mathcal{P}$, contradicting the assumption. Hence, $R \in \mathcal{H}_r$, and there is a graph $H \in \mathcal{H}$ on at most $s=\Psi_2(\mathcal{H},r)$ vertices with $H \rightarrow_{\textrm{ind}} R$. Consider such an induced homomorphism $f:V(H) \rightarrow V(R)$, and let $a_i=|f^{-1}(i)|$. If among the sampled $v$ vertices there are at least $a_i$ vertices from $Q_i$ for each $i$, and these $a_i$ vertices form a clique if $(i,i)$ is a loop in $R$ and otherwise they form an independent set, then $H$ is an induced subgraph of the sampled set of vertices and hence the subgraph induced by the sampled set does not have property $\mathcal{P}$.

We first estimate the probability of the event that the sampled set of $v$ vertices contains $a_i$ vertices in $Q_i$ for each $i$. For a particular $i$, the probability that the sampled set of $v$ vertices contains fewer than $a=a_i$ vertices in $Q_i$ is $0$ if $a=0$ and is otherwise less then \begin{eqnarray*}(1-\delta)^{v-a}{v \choose a} < e^{-\delta v/2}v^a < 1/(8s),\end{eqnarray*}
where we used that $a<s$ and $v$ can be chosen so that $v>10(s/\delta)^2$.
Taking the union bound and summing over all $i$, the probability that the sampled $v$ vertices do not contain $a_i$ vertices in $Q_i$ for some $i$ is at most $s \times 1/(8s)=1/8$.

We now condition on the event that we have at least $a_i$ vertices chosen from $Q_i$ for each $i$. The probability that these $a_i$ vertices do not form a clique if $(i,i)$ is a loop and an independent set if $(i,i)$ is not a loop is at most ${a_i \choose 2}\alpha$. Summing over all $i$, the probability that, for every $i$, the $a_i$ vertices in $Q_i$ form a clique if $(i,i)$ is a loop and an independent set if $(i,i)$ is not a loop, is at least $$1-\sum_i {a_i \choose 2}\alpha \geq 1-{s \choose 2}\alpha \geq 7/8.$$
Hence, with probability at least $3/4$, the induced subgraph on the sampled set of $v$ vertices has the desired properties, which completes the proof.  $\hfill\square$

\smallskip

Our goal for the rest of the subsection is to prove Theorem \ref{reg7}. We first prove a Ramsey-type lemma which states that semi-algebraic graphs contain large balanced complete or empty $h$-partite subgraphs.

\begin{lemma}\label{lemmacompletepartite}
For every $d,t,$ and $D$, there is a constant $c=c(d,t,D)$ satisfying the following condition. For any positive integer $h$, any semi-algebraic graph $G = (P,E)$ in $d$-space with complexity $(t,D)$ has vertex subsets $A_1,\ldots,A_h$ with $|A_1|=\cdots=|A_h| \geq h^{-c}|P|$ such that every pair $A_i,A_j$ with $i \not = j$ is complete or none of them are.
\end{lemma}
\begin{proof}
For $h=1$, the result is trivial by taking $A_1=P$. Thus, we may assume $h \geq 2$. It is shown in \cite{noga} that there is a constant $C  = C(d,t,D)$ such that every induced subgraph of $G$ on $h^C$ vertices contains a clique or an independent set of order $h$.  Applying Theorem \ref{reg1} with $\epsilon=\frac{1}{2h^C}$, we obtain an equitable partition with $\epsilon^{-O(1)}$ parts such that all but at most a $\frac{1}{2h^C}$-fraction of the pairs of parts are homogeneous. Applying Tur\'an's theorem to the auxiliary graph with a vertex for each part and an edge between each homogeneous pair, we obtain $h^C$ parts that are pairwise homogeneous.  Picking one vertex from each of these parts, we obtain an induced subgraph of $G$ on $h^C$ vertices, and by the discussion above, there is an induced subgraph with $h$ vertices which is complete or empty. The parts these vertices come from (after possibly deleting a vertex from some parts to guarantee that they have the same size) have the desired properties.
\end{proof}

We next prove Theorem \ref{reg7}, a strengthening of our quantitative semi-algebraic regularity lemma, via three applications of Theorem \ref{reg1}.

\smallskip

\noindent {\it Proof of Theorem \ref{reg7}.}
We will apply Theorem \ref{reg1} three times. We first apply Theorem \ref{reg1} to obtain a partition $P = P_1\cup \cdots \cup P_K$ with $K=\epsilon^{-O(1)}$ with the implied constant depending on $d,t,$ and $D$, such that all but an $\epsilon$-fraction of the pairs $P_i,P_j$ are homogeneous. We apply Theorem \ref{reg1} again (or rather its proof) to get a refinement with approximation parameter $\epsilon'=1/K^4$, so that all but an $\epsilon'$-fraction of the pairs of parts are homogeneous. Thus, with a positive probability, a random choice of parts $W_1,\ldots,W_K$ of this refinement with $W_i \subset P_i$ has the property that $W_i,W_j$ is homogeneous for all $i \not = j$.  Indeed, at most a fraction $K^2\epsilon'=1/K^2$ of the pairs $W_i \subset P_i$, $W_j \subset P_j$ are not homogeneous, and so, by linearity of expectation, this probability is at least $1-{K \choose 2}\frac{1}{K^2}>1/2>0$.  From Lemma \ref{lemmacompletepartite},  applied to the subgraph induced by $W_i$ for each $i$ with $h=2/\alpha$, we obtain subsets $Q_i \subset W_i$ such that $$|Q_i| \geq \left(\frac{\alpha}{2}\right)^{O(1)}|W_i| \geq \left(\frac{\alpha}{2}\right)^{O(1)}(K^{-4})^{O(1)}|P_i| \geq \delta |P|,$$ with $\delta=(\alpha \epsilon)^c$ for an appropriate choice of $c=c(d,t,D)$, each $Q_i$ is a complete or empty balanced $h$-partite graph, so that the density in $Q_i$ is at most $\alpha$ or at least $1-\alpha$.  This completes the proof.~$\hfill\square$

\subsection{Testing hereditary properties of hypergraphs}

We next state and prove the hereditary property testing result for semi-algebraic hypergraphs.

Let $R$ be a $k$-uniform hypergraph with vertex set $[r']$, and $B$ be a {\em blow-up} of $R$ with vertex sets $V_1,\ldots,V_{r'}$. That is, $B$ is a $k$-uniform hypergraph on $V_1 \cup \ldots \cup V_{r'}$, where $(v_1,\ldots,v_k) \in (V_{i_1},\ldots,V_{i_k})$ is an edge if and only if $(i_1,\ldots,i_k)$ form an edge of $R$. An extension of $B$ (with respect to $V_1,\ldots,V_{r'}$) is any hypergraph on $V_1 \cup \ldots \cup V_{r'}$ which agrees with $B$ on the $k$-tuples with vertices in distinct $V_i$.

For a hypergraph $H$, we say that $R$ is {\it extendable $H$-free} if each blow-up of $R$ has an extension which contains no induced copy of $H$.  For a family $\mathcal{H}$ of hypergraphs, we say that $R$ is {\it extendable $\mathcal{H}$-free} if each blow-up of $R$ has an extension which contains no induced $H \in \mathcal{H}$.

For a hypergraph property $\mathcal{P}$, we say that $R$ {\em strongly has property} $\mathcal{P}$ if every blow-up of $R$ has an extension which has property $\mathcal{P}$. Otherwise, there are a smallest $s=s(\mathcal{P},R)$ and vertex sets $V_1,\ldots,V_{r'}$ with $r' \leq r$ and $s(\mathcal{P},R)=|V_1|+\cdots+|V_{r'}|$ such that no extension of the blow-up $B$ of $R$ with vertex sets $V_1,\ldots,V_{r'}$ has property $\mathcal{P}$.

 Define $\Psi_3(\mathcal{P},r)$ to be the maximum of $s(\mathcal{P},R)$ over all $R$ with at most $r$ vertices which do not strongly have property $\mathcal{P}$.

Our next theorem is about hereditary property testing for semi-algebraic hypergraphs. It implies that, if $\Psi_3(\mathcal{H},r)$ is at most polynomial in $r$, then $\mathcal{P}$ can be easily tested within the semi-algebraic hypergraphs of constant description complexity, i.e., there is an $\epsilon$-tester with vertex query complexity $\epsilon^{-O(1)}$. A simple example in which $\Psi_3(\mathcal{H},r)$ is constant is the case that the property $\mathcal{P}$ is the family of hypergraphs which are induced $H$-free for some fixed hypergraph $H$.

\begin{theorem}
\label{hypergraphprop2}
Let $\mathcal{P}$ be a hereditary property of hypergraphs. Within the family $\mathcal{A}$ of semi-algebraic $k$-uniform hypergraphs in $d$-space with description complexity $(t,D)$, the property $\mathcal{P}$ can be $\epsilon$-tested with vertex query complexity at most $r^C\Psi_3(\mathcal{H},r)^C$, where $r=(1/\epsilon)^C$ with $C=C(d,t,D)$.
\end{theorem}
We will need the following polynomial strong regularity lemma for semi-algebraic hypergraphs.

\begin{lemma}\label{reg23}
For any $0< \epsilon < 1/2$, any semi-algebraic $k$-uniform hypergraph $H = (P,E)$ in $d$-space with complexity $(t,D)$ has an equitable vertex partition $P = P_1\cup \cdots \cup P_{r'}$ such that all but an $\epsilon$-fraction of the $k$-tuples of distinct parts are homogeneous. Furthermore, there are subsets $Q_i \subset P_i$ for each $i$ such that every $k$-tuple of distinct parts is homogeneous and $|Q_i| \geq \delta |P|$ with $\delta=\epsilon^{c}$, where $c=c(k,d,t,D)$.
\end{lemma}
\begin{proof}
The proof follows the graph case, as in Theorem \ref{reg7}, and involves two applications of Theorem \ref{reg1}. First, we apply Theorem \ref{reg1} to obtain a partition $P = P_1\cup \cdots \cup P_K$ with $K=\epsilon^{-O(1)}$, where the implied constant depends on $k,d,t,D,$ such that all but an $\epsilon$-fraction of the $k$-tuples of parts are homogeneous. We apply Theorem \ref{reg1} (or rather its proof) again to get a refinement with approximation parameter $\epsilon'=1/K^{2k}$, so that all but an $\epsilon'$-fraction of the $k$-tuples of parts are homogeneous. Thus, with positive probability, a random choice of parts $Q_1,\ldots,Q_K$ of this refinement with $Q_i \subset P_i$ has the property that each $k$-tuple $Q_{i_1},\ldots,Q_{i_k}$ of distinct parts is homogeneous. This completes the proof.
\end{proof}

\noindent {\it Proof of Theorem \ref{hypergraphprop2}.} Let $s=\Psi_3(\mathcal{P},r)$ and $v=(rs)^C$ for an appropriate constant $C=C(k,d,t,D)$. Consider the tester which samples $v$ vertices from a graph $A=(P,E) \in \mathcal{A}$. It accepts if the induced subgraph on these $v$ vertices has property $\mathcal{P}$ and rejects otherwise. It suffices to show that if $A \in \mathcal{A}$ is $\epsilon$-far from satisfying $\mathcal{P}$, then with probability at least $2/3$, the tester will reject.

Consider an equitable partition $P=P_1 \cup \ldots \cup P_{r'}$ of the vertex set of $A$ guaranteed by Lemma \ref{reg23} so that all but at most an $\epsilon$-fraction of the pairs of parts are homogeneous, and subsets $Q_i \subset P_i$ such that every $k$-tuple of distinct $Q_i$ is homogeneous, and $|Q_i| \geq \delta |P|$ for each $i$ with $\delta=\epsilon^{c}$, where $c=c(k,d,t,D)$. Let $R$ be the $k$-uniform hypergraph on $[r']$, where a $k$-tuple $(i_1,\ldots,i_k)$ of distinct vertices forms an edge if and only if $Q_{i_1},\ldots,Q_{i_k}$ is complete. If $R$ strongly has property $\mathcal{P}$, then $A$ is $\epsilon$-close to a hypergraph which has property $\mathcal{P}$, and hence the algorithm accepts in this case. Thus, we may assume that $R$ does not strongly have property $\mathcal{P}$ so that there are sets $V_1,\ldots,V_{r'}$ with $|V_1|+\cdots+|V_{r'}| = s(\mathcal{P},R) \leq \Psi_3(\mathcal{P},r)=s$ such that every extension of the blow-up of $R$ with parts $V_1,\ldots,V_{r'}$ does not have property $\mathcal{P}$. Thus, if among the $v$ sampled vertices, we get for every $i$ at least $|V_i|$ vertices in $Q_i$, then the subgraph induced by the sampled vertices does not have property $\mathcal{P}$ and the algorithm rejects. Therefore, it suffices to show that with probability at least $2/3$, we get at least $|V_i|$ vertices in each $Q_i$. However, this is the computation we already did in the proof of Theorem \ref{graphprop}. It is sufficient, for example, to assume that $v>10(s/\delta)^2$, and we can take $v$ to satisfy this condition.
$\hfill\square$

\medskip

\noindent \textbf{Acknowledgments.}  We would like to thank the anonymous referees of the conference version \cite{FoPaSu} for their helpful comments, including an improvement of the exponent in the bound in Theorem~\ref{pach}.


\begin{thebibliography}{99}



\bibitem{Al} N.~Alon, \emph{Testing subgraphs in large graphs}, Random Structures Algorithms, \textbf{21} (2002), pp.~359--370.

\bibitem{AlFo} N.~Alon and J.~Fox, \emph{Easily testable graph properties}, Combin.~Probab.~Comput., \textbf{24} (2015), pp.~646--657.


\bibitem{noga} N.~Alon, J.~Pach, R.~Pinchasi, R.~Radoi\v{c}i\'c, and M. Sharir, \emph{Crossing patterns of semi-algebraic sets}, J. Combin. Theory Ser. A, \textbf{111} (2005), pp.~310--326.


\bibitem{AlSh1} N.~Alon and A.~Shapira, \emph{A characterization of easily testable induced subgraphs}, Combin. Probab. Comput., \textbf{15} (2006), pp.~791--805.

\bibitem{ASh} N.~Alon and A.~Shapira, \emph{Every monotone graph property is testable}, SIAM J. Comput., {\bf 38} (2008), pp.~505--522.

\bibitem{AlSh} N.~Alon and A.~Shapira, \emph{A characterization of the (natural) graph properties testable with one-sided error}, SIAM J. Comput., \textbf{37} (2008), pp.~1703--1727.


\bibitem{AT} T.~Austin and T.~Tao, \emph{Testability and repair of hereditary hypergraph properties}, Random Structures Algorithms, \textbf{36} (2010), pp.~373--463.

\bibitem{barany2} I.~B\'ar\'any and J.~Pach, \emph{Homogeneous selections from hyperplanes}, J. Combinat. Theory Ser. B, \textbf{104} (2014), pp.~81--87.

\bibitem{barany} I.~B\'ar\'any and P.~Valtr, \emph{A positive fraction Erd\H os-Szekeres theorem}, Discrete Comput. Geom., \textbf{19} (1998), pp.~335--342.


\bibitem{basu}  S.~Basu, R.~Pollack, and M.~F. Roy, \emph{Algorithms in
Real Algebraic Geometry}, 2nd Edition, Algorithms and
Computation in Mathematics \textbf{10}, Springer-Verlag, Berlin,
2006.



\bibitem{hubard} B.~Bukh and A.~Hubard, \emph{Space crossing numbers}, Combin.~Probab.~Comput., \textbf{21} (2012), pp.~358--373.

\bibitem{Ch93} B.~Chazelle, \emph{Cutting hyperplanes for divide-and-conquer}, Discrete Comput. Geom., 9 (1993), pp.~145--158.

\bibitem{chazelle} B.~Chazelle, H.~Edelsbrunner, L.~Guibas, and M.~Sharir,  \emph{A singly exponential stratification scheme for real semi-algebraic varieties and its applications}, Theor. Comput. Sci., \textbf{84} (1991), pp.~77--105.

\bibitem{chazelle2} B.~Chazelle and J.~Friedman, \emph{A deterministic view of random sampling and its use in geometry}, Combinatorica, \textbf{10} (1990), pp.~229--249.

\bibitem{clark} K.~Clarkson, \emph{A randomized algorithm for closest-point queries}, SIAM J. Comput., \textbf{17} (1988), pp.~830--847.

\bibitem{CoFo} D.~Conlon and J.~Fox, \emph{Bounds for graph regularity and removal lemmas}, Geom. Funct. Anal., \textbf{22} (2012), pp.~1191--1256.

\bibitem{CoFo2} D.~Conlon and J.~Fox, \emph{Graph removal lemmas, Surveys in Combinatorics 2013}, 1--49, London Math. Soc. Lecutre Note Ser., 409, Cambridge Univ. Press, Cambridge, 2013.

\bibitem{conlon} D.~Conlon, J.~Fox, J.~Pach, B.~Sudakov, and A.~Suk, \emph{Ramsey-type results for semi-algebraic relations}, Trans. Amer. Math. Soc., \textbf{366} (2014), pp.~5043--5065.

\bibitem{erdos1} P.~Erd\H os, \emph{On extremal problems of graphs and generalized graphs}, Israel J. Math., \textbf{2} (1965), pp.~183--190.

\bibitem{es} P.~Erd\H os and G.~Szekeres, \emph{A combinatorial problem in geometry}, Compos.~Math., \textbf{2} (1935), pp.~463--470.

\bibitem{gromov} J.~Fox, M.~Gromov, V.~Lafforgue, A.~Naor, and J.~Pach, \emph{Overlap properties of geometric expanders}, J. Reine Angew. Math. (Crelle's Journal), \textbf{671} (2012), pp.~49--83.

\bibitem{FoPaSu} J.~Fox, J.~Pach, and A.~Suk, \emph{Density and regularity theorems for semi-algebraic hypergraphs}, {\it Proceedings of the Twenty-Sixth Annual ACM-SIAM Symposium on Discrete Algorithms}, 1517-1530, SIAM, San Diego, California, 2015.

\bibitem{FoLo} J.~Fox and L. M.~Lov\'asz, \emph{A tight lower bound for Szemer\'edi's regularity lemma}, Combinatorica, to appear.

\bibitem{GGR} O.~Goldreich, S.~Goldwasser, and D.~Ron, \emph{Property testing and its applications to learning and approximation}, J. ACM, \textbf{45} (1998), pp.~653--750.

\bibitem{gp} J.~E.~Goodman and R.~Pollack, \emph{Allowable sequences and order-types in discrete and computational geometry}, In J.~Pach editor, New Trends in Discrete and Computational Geometry, \textbf{10} (1993), Springer, Berlin etc., pp.~103--134.

\bibitem{pollack} J.~E.~Goodman and R.~Pollack, \emph{The complexity of point configurations}, Discrete Appl. Math., \textbf{31} (1991), pp.~167--180.

\bibitem{Go1} W.~T.~Gowers, \emph{Hypergraph regularity and the multidimensional Szemer\'edi theorem}, Ann. of
Math., \textbf{166} (2007), pp.~897--946.

\bibitem{Go97} W.~T.~Gowers, \emph{Lower bounds of tower type for Szemer\'edi's uniformity lemma}, Geom. Funct. Anal., {\bf 7} (1997), pp.~322--337.

\bibitem{Go2} W.~T.~Gowers, \emph{Quasirandomness, counting and regularity for 3-uniform hypergraphs}, Combin.
Probab. Comput., \textbf{15} (2006), pp.~143--184.

\bibitem{hubard2} A.~Hubard, L.~Montejano, E.~Mora, and A.~Suk, \emph{Order types of convex bodies}, Order, \textbf{28} (2011), pp.~121--130.

\bibitem{kalai} G.~Kalai, \emph{Intersection patterns of convex sets}, Israel J. Math., \textbf{48} (1984), pp.~161--174.

\bibitem{kar} R.~Karasev, \emph{A simpler proof of the Boros-F\"uredi-B\'ar\'any-Pach-Gromov theorem}, Discrete Comput. Geom., \textbf{47} (2012), pp.~492--495.

\bibitem{kyncl} R.~Karasev, J.~Kyn\v{c}l, P.~Pat\'ak, Z.~Pat\'akov\'a, M.~Tancer, \emph{Bounds for Pach's selection theorem and for the minimum solid angle in a simplex}, Discrete Comput. Geom., \textbf{54} (2015), pp.~610--636.

\bibitem{sz} J.~Koml\'os and M.~Simonovits, \emph{Szemer\'edi's regularity lemma and its applications in graph theory}.  In D.~Miklos et al. editors, Combinatorica, Paul Erd\H os Is Eighty, \textbf{2} (1996), pp.~295--352.

\bibitem{KRT} Y.~Kohayakawa, V.~R\"odl, L.~Thoma, \emph{An optimal algorithm for checking regularity}, SIAM J.~Comput., \textbf{32} (2003), pp.~1210--1235.

\bibitem{KST} T.~K\H ov\'ari, V.~T.~S\'os, and P.~Tur\'an, \emph{On a problem of K. Zarankiewicz}, Colloquium Math., \textbf{3} (1954), pp.~50--57.



\bibitem{matousek} J.~Matou\v{s}ek, \emph{Lectures on Discrete Geometry}, Springer-Verlag New York, Inc., 2002.


\bibitem{MoSh}
G.~Moshkovitz and A.~Shapira, \emph{A short proof of Gowers' lower bound for the regularity lemma}, Combinatorica, to appear.

\bibitem{NRSS} B.~Nagle, V.~R\"odl and M.~Schacht, \emph{The counting lemma for regular k-uniform hypergraphs},
Random Structures Algorithms, \textbf{28} (2006), pp.~113--179.

\bibitem{pach2} J.~Pach, \emph{A Tverberg-type result on rainbow simplices}, Computational Geometry, \textbf{10} (1998), pp.~71--76.

\bibitem{pach} J.~Pach and P.~Agarwal, \emph{Combinatorial Geometry}, New York, Wiley, 1995.

\bibitem{PS5} J.~Pach and J.~Solymosi, Crossing patterns of segments, J. Combin. Theory Ser. A \textbf{96} (2001), 316--325.

\bibitem{rodlschacht} V.~R\"odl and M.~Schacht, \emph{Property testing in hypergraphs and the removal lemma}, Proceedings of the 39th Annual ACM Symposium on Theory of Computing (STOC 2007), ACM, New York, 2007, pp.~488--495.


\bibitem{RoSc} V.~R\"odl and M.~Schacht, \emph{Generalizations of the removal lemma}, Combinatorica, \textbf{29} (2009), pp.~467--501.

\bibitem{RuSu} R.~Rubinfeld and M.~Sudan, \emph{Robust characterization of polynomials with applications to program testing}, SIAM J. on Computing, \textbf{25} (1996), pp.~252--271.

\bibitem{szemeredi} E.~Szemer\'edi, \emph{Regular partitions of graphs},   Probl\`emes combinatoires et th\'eorie des graphes (Colloq. Internat. CNRS, Univ. Orsay), \textbf{260}, CNRS, Paris, 1978, pp.~399--401.

\bibitem{tverberg} H.~Tverberg, \emph{A generalization of Radon's theorem}, J. Lond. Math. Soc., \textbf{41} (1966), pp.~123--128.

\bibitem{color} R.~T.~\v{Z}ivaljevi\'c and S.~T.~Vre\'cica, \emph{The colored Tverberg's problem and complexes of injective functions}, J. Combin. Theory Ser. A, \textbf{61} (1992), pp.~309--318.





\end{thebibliography}
 \end{document}